\documentclass[a4paper,11pt,twoside]{article}

\usepackage[utf8]{inputenc}
\usepackage[T1]{fontenc}
\usepackage[english]{babel}

\usepackage{charter}

\usepackage{indentfirst}
\usepackage{xcolor}
\usepackage{verbatim}

\usepackage{hyperref}

\usepackage{pifont}

\usepackage[top=3.cm, bottom=4.0cm, left=2.2cm, right=2.2cm]{geometry}
\usepackage{amsmath}
\usepackage{amsthm}	
\usepackage{amsfonts}	
\usepackage{amssymb	
           ,bbm
            ,units 
           }
\usepackage{enumerate}
\usepackage[nobysame
           ,alphabetic
           ,numeric
           ,initials
           ]{amsrefs}

\usepackage{
            ulem		
           ,soul		
} 

\numberwithin{equation}{section}
\numberwithin{figure}{section}
 \usepackage[nodayofweek]{datetime}

\renewcommand*{\thefootnote}{\fnsymbol{footnote}}

\title{It\^o-Wentzell-Lions formula for measure dependent random fields under full and conditional measure flows}
\author{
 	\normalsize Gon\c calo dos Reis\footnote{G. dos Reis acknowledges support from the \emph{Funda{\c c}$\tilde{\text{a}}$o para a Ci$\hat{e}$ncia e a Tecnologia} (Portuguese Foundation for Science and Technology) through the project UIDB/00297/2020 and UIDP/00297/2020 (Centro de Matem\'atica e Aplica\c c\~oes CMA/FCT/UNL).}
 	\\[8pt]
         \small  University of Edinburgh\\ 
         \small  School of Mathematics \\
         \small  Edinburgh, EH9 3FD, UK\\  
         \small  and \\
	\small  Centro de Matem\'atica e Aplica\c c\~oes \\
	\small (CMA), FCT, UNL, Portugal \\
        \small  G.dosReis@ed.ac.uk \\
        \small Phone: +44 (0) 131 651 7677
 \and
\normalsize Vadim Platonov\\[8pt] 
         \small  University of Edinburgh\\ 
         \small  School of Mathematics \\
         \small  Edinburgh, EH9 3FD, UK\\    
        \small    \\
		\small  \\
          \\
        \small V.d.Platonov@sms.ed.ac.uk
}

\date{ \currenttime, \ddmmyyyydate\today\qquad{(File: \tt \jobname.tex})}

\theoremstyle{plain}
\newtheorem{theorem}{Theorem}[section]

\newtheorem{proposition}[theorem]{Proposition}

\newtheorem{definition}[theorem]{Definition}

\newtheorem{remark}[theorem]{Remark}

\newtheorem{assumption}[theorem]{Assumption}

\newcommand{\bE}{\mathbb{E}}
\newcommand{\bF}{\mathbb{F}}

\newcommand{\bN}{\mathbb{N}}
\newcommand{\bP}{\mathbb{P}}

\newcommand{\bR}{\mathbb{R}}

\newcommand{\cB}{\mathcal{B}}
\newcommand{\cC}{\mathcal{C}}

\newcommand{\cF}{\mathcal{F}}

\newcommand{\cP}{\mathcal{P}}

\newcommand{\trace}{\textrm{Trace}}
\newcommand{\Supp}{\textrm{Supp}}
\newcommand{\Law}{\textrm{Law}}

\definecolor{darkgreen}{rgb}{0,0.35,0}

\newcommand{\1}{\mathbbm{1}}

\newcommand{\ti}{{t_i}}
\newcommand{\tip}{{t_{i+1}}}

\newcommand\dd{\text{d}}

\begin{document}

\selectlanguage{english}

\maketitle
\renewcommand*{\thefootnote}{\arabic{footnote}}

\vspace{-1cm}
\begin{abstract} 
We present several It\^o-Wentzell formulae on Wiener spaces for real-valued functional random field of It\^o type that depend on measure flows. We distinguish the full- and the marginal-measure flow cases in the spirit of mean-field games.  Derivatives with respect to the measure components are understood in the sense of Lions.
\end{abstract}
\vspace{0.3cm}


%

%
%
%
\footnotesize
\setcounter{tocdepth}{2}
\tableofcontents
\normalsize
\footnotesize


\normalsize
%
%
%

\section{Introduction}

The extension of the celebrated It\^o chain rule from deterministic regular functions to random fields of It\^o type was proposed originally by Wentzell \cite{ventzel1965equations} and later  generalised in \cite{rozovskii1990stochastic,kunita1997stochastic,krylov2011ito,kunita1981some,FlandoliRusso2002}. 
This successful result has appeared in SPDE problems from wellposedness to numeric methods and applications to fluid dynamics modelling \cites{rozovskii1990stochastic,carmona1999stochastic,GerencserGyoengy2017,FlandoliRusso2002,de2019implications}, in stochastic regularisation problems \cite{DuboscqReveillac2016}, filtering \cite{KrylovWang2011}, and mathematical Finance \cites{nicole2013exact,AhmadHamblyLedger2018,matoussi:hal-03025475}. 

More recently, there has been an explosion of new literature around mean-field models where, thinking of stochastic differential equations, maps now contain a dependency on the measure of the solution map. These are the  so-called mean-field equations or McKean-Vlasov equations and have appeared as models in many different sub-fields. A critical tool here is the so-called It\^o-Lions formula that extends the classical It\^o formula and allows one to produce a dynamics for functionals of measure flows \cite{CarmonaDelarue2017book1,CarmonaDelarue2017book2}. The construction relies on the notion of Lions derivative for measure functionals.

Regarding measure derivatives, we point out that the earlier (to Lions \cite{lions2007cours}) notion of intrinsic (measure) derivative was introduced in \cite{AlbeverioKondratievRockner1996}, and used for stochastic analysis on the configuration space over Riemannian manifolds -- see additionally  \cite{RenWang2021DerivativeFormulas} for characterisations on the relations of different type derivatives in measures. The more recent Lions derivative is a stronger notion than the intrinsic derivative of \cite{AlbeverioKondratievRockner1996} in the sense that, over the same space, the intrinsic derivative is a Gateaux derivative while the Lion derivative is a Fr\'echet derivative. We remark that the Lions derivative concept is not an intuitive notion especially when seen from the lens of geometric analysis (but the intrinsic is); arguably, the intrinsic derivative is not the intuitive concept when seen from the lens of Mathematical finance. 
Lastly, we emphasize that both these derivative concepts are different from the ``linear functional derivative'' one usually sees in optimal transport e.g., \cite{AmbrosioGigliSavare2005book} -- see ``Otto Calculus'' in  \cite[Ch 15]{Villani2009} and for a comparison \cite[Section 5.2]{CarmonaDelarue2017book1}.


For \emph{deterministic} functionals of measures, extending the classical It\^o formula to the so-called It\^o-Lions formula, there are several approaches and results available in the literature. The classical difference of increments approach is used in \cite{buckdahn2017mean} under a strong regularity assumption of existence of second order Fr\'echet derivatives. In  \cite{chassagneux2014classical} an approach using projections over empirical measures is used allowing for minimal regularity assumptions. Both approaches are neatly reviewed in \cite{CarmonaDelarue2017book1}*{Chapter 5}. 
Linked to the existence of a regular solution to the master equation for mean-field games with common noise is the approach by \cite{carmona2014master}*{Appendix 6}. Their proof is carried out using It\^o-Taylor type expansions (similar to \cites{buckdahn2017mean}) and requiring the involved maps to be twice Fr\'echet differentiable. 
Lastly, another approach is to use a semi-group type approach to describe the flow of measures and obtain the necessary infinitesimal expansions see \cite{cardaliaguet2019master}*{Appendix A}. More recently  \cite{cavallazzi2021krylov} present such It\^o-Lions formula for maps belonging to Sobolev spaces, \cite{guo2020s,talbi2021dynamic} provide also such formula for semi-martingales -- these three works leave out the conditional measure-flow case. An It\^o-type formula for measure-valued diffusion processes as conditional distributions of image dependent SDEs has been proved in \cite{Wang2021Imagedependent}.
To the best of our knowledge we have found only one It\^o-Wentzell-Lions type formula in the literature, \cite{cardaliaguet2019master}*{Appendix A}. Their approach is set in relation to an existing regular solution to a certain master equation for mean-field control games with common noise. Their proof is carried out via expansions of the densities of the underlying (conditional) measure flow but where the involved diffusion components are constants.  

\smallskip
\emph{Our contribution.} In this manuscript we propose It\^o-Wentzell formulae for random fields that embed measure-functionals in a way that is amenable to an analysis in the sense of Lions derivatives. We establish two formulae, and two further corollaries, all decoupled from the applications either in mean-field game theory in finance \cite{CarmonaDelarue2017book1,CarmonaDelarue2017book2,guo2020s,talbi2021dynamic}, fluid mechanics \cite{holm20StochasticEffectsOfWaves,kolokoltsov2017RegularityAndSensitivity,BossyJabirTalay2011OnConditional}, neuroscience modeling \cite{Erny2021ConditoinalPoC}, population dynamics models \cite{bayraktar21finite-space} or further related stochastic analysis problems \cite{Lacker2019InvertingTM,platonov2021StratonovichIto} albeit motivated by them. 

Our first result is for the full flow of measures (the measure is deterministic) while the second is for a partial flow of measures (the measure is random). Each result is then further extended to a full joint chain rule allowing for the an additional driving stochastic processes $(X_t)_{t \geq 0}$ having a semi-martingale expansion. In particular, we recover the results in \cite{cardaliaguet2019master}*{Appendix A} while finessing their assumptions, see our Remark \ref{rem:SlightDifference2CDLL2019} below. A by-product of our results is a clarification on the necessity of the assumptions on the classical It\^o-Wentzell formula \cite{kunita1997stochastic}*{Theorem 3.3.1} (see our Theorem \ref{theo-Classic-Ito-Wentzell} below). Namely, we prove that one can require one order of regularity less from the drift and diffusion coefficient of the random vector field to which the It\^o-Wentzell formula is applied to (see our Theorem \ref{theo:ExtendItoWentzell}). This smaller result is of its own interest.

The usefullness of these result is manyfold. Direct applications within mean-field optimal control could be envisaged in neuroscience modeling \cite{Erny2021ConditoinalPoC}; extending the contribution of \cite{elkaroui2018consistent,matoussi:hal-03025475}, where the classical It\^o-Wentzell formula is used to develop a consistent forward utilities of investment and consumption -- introducing the relative performance concerns (as in \cite{Platonov2020forward,platonov2020forwardCRRA}). 
Also building from \cite{bayraktar21finite-space}, where a  mean-field games with Fisher-Wright common noise is discussed. This model is used in the evolution of population genetics and where it would be natural to update the model to support the distributional component, making use of the results we provide in order to establish the verification procedure. 
In fluid dynamics these formulae would allow to expand the dynamics of driving signals against the underlying vector field
\cite{holm20StochasticEffectsOfWaves,kolokoltsov2017RegularityAndSensitivity,BossyJabirTalay2011OnConditional}. 

Lastly, our work can be extended in several directions to include anticipative processes  \cite{ocone1989generalized}, general semimartingale dynamics \cite{talbi2021dynamic,guo2020s}, path dependent functionals in combination with functional It\^o calculus \cite{cont2013functional}, or extensions to $K$-forms for SPDEs in fluid dynamics \cite{de2019implications}.

\smallskip 
\textit{Methodological perspective: from It\^o-Lions and It\^o-Wenzell to It\^o-Wenzell-Lions.} Our proofs combine two techniques, the projection over empirical measures approach of \cites{chassagneux2014classical,CarmonaDelarue2017book1}, which have the benefit of yielding lower regularity requirements on the underlying coefficients and Taylor-like expansion arguments in the vein of \cite{kunita1981some} -- we argue next that this is the suitable methodology for this result. 

The chain rule in the measure component first appears in \cite{buckdahn2017mean} making use of the telescopic summation technique and building on a strong assumption of a second order Fr\'echet differentiability of the lifting map. 
To overcome the requirement of a second Fr\'echet derivative and reduce it to just first order Fr\'echet derivative (in fact the so-called \emph{Partial-$\cC^2$} regularity)  for full measure case, the approach of empirical projection was introduced  \cite{chassagneux2014classical,CarmonaDelarue2017book1,CarmonaDelarue2017book2}:  
this is the approach we follow. 
In \cite{cardaliaguet2019master} the It\^o-Wentzell-Lions formula is shown under the constant diffusion of the random field. The authors follow the semi-group approach and require the existence of the density.

Recently, \cite{guo2020s} introduced the use of cylindrical polynomials approximation to build a measure chain rule for the measure flow of semimartingales, i.e., an It\^o-Lions formula for semimartingales. 
Finally, \cite{talbi2021dynamic} shows an It\^o-Lions formula for semimartingales with jumps (the exact same result of \cite{guo2020s}) but using the mechanisms of \cite{buckdahn2017mean}. Concretely, they make use of a telescopic summation technique building on the functional linear derivative instead of the Lions one. 
This approach relies on the assumptions of growth and boundedness of the functional linear derivative and its partial derivative with respect to new spatial variable.  
For both \cite{guo2020s,talbi2021dynamic} the conditional measure flow case is left unaddressed.

We already argued that neither the proof techniques of \cite{cardaliaguet2019master} or \cite{buckdahn2017mean} are appropriate as proofs for our results. The former requires constant diffusion coefficients to ensure existence of densities while the latter requires higher Fr\'echet regularity than needed. Hence the reason we follow \cite{chassagneux2014classical,CarmonaDelarue2017book1,CarmonaDelarue2017book2}. Two recent works \cite{guo2020s,talbi2021dynamic}, posterior to ours, use new techniques to prove the It\^o-Lions formula for general semimartingales (\textit{for deterministic fields}) --- it is not clear if those techniques can be adapted to prove the It\^o-Wentzell-Lions formulae we present in this manuscript under the same minimal regularity constraints we impose. The difficulty stems from our use of random fields while in \cite{guo2020s,talbi2021dynamic} the fields are deterministic. 

Concretely, to prove the It\^o-Wentzell-Lions formula of this manuscript with the same methodology of \cite{guo2020s} one would require a Leibniz rule to interchange the Fr\'echet derivative symbol with the stochastic integral one and thus would demand further regularity assumptions on top of the existing ones -- a general Leibniz rule within this framework is presently an open question\footnote{Preliminary work on exchanging Fr\'echet derivatives with the Lebesgue's integral has been carried in a note by O.~Kammar \cite{Kammar2016FrechetLebesgue}. It is unclear presently under which minimal conditions can one exchange Fr\'echet derivatives (and later the Lions  derivatives) with the stochastic (It\^o) integral.}. 
Moreover, such a result is not needed in  \cite{guo2020s,talbi2021dynamic} due to their use of deterministic fields! 
The telescopic summation approach from \cite{talbi2021dynamic} has another limitation in the context of proving an It\^o-Wentzell-Lions formula. 
One needs to expand the local difference of integrands by the application of the classical It\^o-Wentzell formula which requires the existence and well-definiteness of the random fields spanned by the differentiation in measure (in sense of Lions or linear functional; see our Theorem \ref{theo-Classic-Ito-Wentzell}).  
This limitation could be avoided by proving the aforementioned Leibniz rule which in turn would demand stronger regularity for the random field and its characteristics as mentioned earlier. For these reasons, we argue that the `empirical projection' technique \cites{chassagneux2014classical,CarmonaDelarue2017book1,CarmonaDelarue2017book2} is the suitable methodology.

\medskip
\textbf{Organisation of the paper.} In Section 2 we set notation and review a few concepts necessary for the main constructions. In Section \ref{sec:fullflow} we state the full measure flow results. While Section 2 builds towards Section \ref{sec:fullflow}, we will need to reframe some notation for Section \ref{sec:conditionalflow} where we present the conditional flow results. 
\smallskip

\textbf{Acknowledgements.} The authors would like to thank Fran\c cois Delarue (Universit\'e de Nice Sophia-Antipolis, FR) for the helpful discussions. 
\smallskip

\textbf{Data access statement.} No data was generated by this project.

%
%
%

%
%
%
\section{Notation and auxiliary results}
\label{sec:two}

\subsection{Notation and Spaces}

Let $\bN$ be the set of natural numbers starting at $1$, $\bR$ denotes the real numbers. For  collections of vectors in $\{x^l\}_{l}\in \bR^d$, let the upper index $l$ denote the distinct vectors, whereas the lower index the vector components, i.e. $x^l=(x^l_1,\cdots,x^l_d)\in \bR^d$ namely $x^l_j$ denotes the $j$-th component of $l$-th vector. For $x,y \in \bR^d$ denote the scalar product by $x \cdot y=\sum_{j=1}^d x_j y_j$; and $|x|=(\sum_{j=1}^d x_j^2)^{1/2}$ the usual Euclidean distance; and $x \otimes y$ denotes the tensor product of vectors $x,y \in \bR^d$. Let $\1_A$ be the indicator function of set $A\subset \bR^d$. For a matrix $A \in \bR^{d\times n}$ we denote by $A^\intercal$ its transpose and its Frobenius norm by $|A|=\trace\{A A^\intercal\}^{1/2}$. Let $I_d:\bR^d\to \bR^d$ be the identity map. 

We denote by $\cC(A,B)$ for $A,B \subseteq \bR^d$, $d\in \bN$, the space of continuous functions $f:A\to B$.  In terms of derivative operators and differentiable functions, $\partial_t$ denotes the partial differential in the time parameter $t \in [0,T]$; $\partial_x$ denotes the gradient operators in the spatial variables $x$ in $\bR^d$ while $\partial^2_{xx}, \partial_{yy}^2$ the Hessian operator in $x \text{ or }y \in\bR^d$. 

For $p,d,m\in \bN$ denote $\cC^p(\bR^d,\bR^m)$ the space of $p$-times continuously differentiable functions from $\bR^d$ to $\bR^m$. The space $\cC^1(\bR^d,\bR^m)$ is equipped with a collection of seminorms $\{ \|g\|_{\cC^1(K)} := \sup_{x \in K} ( |g(x)| + |\partial_x g(x)|)$, $g\in \cC^p(\bR^d) \}$, indexed by the compact subsets $K \subset \bR^d$.  The space $\cC^2(\bR^d)$ is equipped with a collection of seminorms $\{ \|g\|_{\cC^2(K)} := \sup_{x \in K} ( |g(x)| + |\partial_x g(x)|+|\partial^2_{xx} g(x)|), g\in \cC^p(\bR^d) \}$, indexed by the compact subsets $K \subset \bR^d$; we refer to $C^{1,2}=C^{1,2}([0,T]\times\bR^d,\bR^m)$ as the usual space of maps $f:[0,T]\times\bR^d\to \bR^m$ that are once continuously differentiable in the first variable, twice so in the second variable (as in $C^2(\bR^d,\bR^m)$) and jointly continuous across the several derivatives.

We say that the function is locally bounded, when its restriction to the compact set is bounded.

\subsection*{Spaces}

We introduce over $\bR^d$ the space of probability measures $\cP(\bR^d)$ and its subset $\cP_2(\bR^d)$ of those with finite second moment. The space $\cP_2(\bR^d)$ is Polish under the Wasserstein distance
\begin{align*}
W_2(\mu,\nu) = \inf_{\pi\in\Pi(\mu,\nu)} \Big(\int_{\bR^d\times \bR^d} |x-y|^2\pi(\dd x,\dd y)\Big)^\frac12, \quad \mu,\nu\in \cP_2(\bR^d) ,
\end{align*}   
where $\Pi(\mu,\nu)$ is the set of couplings for $\mu$ and $\nu$ such that $\pi\in\Pi(\mu,\nu)$ is a probability measure on $\bR^d\times \bR^d$ such that $\pi(\cdot\times \bR^d)=\mu$ and $\pi(\bR^d \times \cdot)=\nu$. Let $\Supp(\mu)$ denote the support of $\mu \in \cP(\bR^d)$.

Throughout set some $0<T<+\infty$ and we work the finite time interval $[0,T]$. Let our probability space be a completion of $(\Omega, \bF, \cF,\bP)$ with $\bF=(\cF_t)_{t \in [0,T]}$ carrying a $d$-dimensional Brownian motion $W=(W^1,\cdots,W^d)$ generating the probability space's filtration, augmented by all $\bP$-null sets, and with an additionally sufficiently rich sub $\sigma$-algebra $\cF_0$ independent of $W$. Let our probability space be an atomless Polish. We denote by $\bE[~\cdot~]=\bE^\bP[~\cdot~]$ the usual expectation operator wrt to $\bP$.   

We adopt the following convention, that for $d$-dimensional random vector $X=(X_1,\cdots,X_d)$ we understand denote $\bE[X]$ by the $d$-dimensional vector $(\bE[X_1],\cdots,\bE[X_d])$. The convenience of this notation will become apparent in the later Section \ref{sec:conditionalflow}.

 We define $L^2(\Omega,\cF_0,\bP,\bR^d)$ as the space of $\cF_0$-measurable random variables $\xi:\Omega\to \bR^d$ that are square integrable $\bE^\bP[|\xi|^2]<\infty$. Given two processes $(X_t)_{t \in [0,T]}$ and $(Y_t)_{t \in [0,T]}$ let $\langle X,Y \rangle_t$ denote their cross-variation up to time $t\in[0,T]$.

Lastly, for convenience we choose to work over $1$-, $d$- and $d\times d$-dimensional spaces. This is particularly helpful in lowering complexity of the presentation of the later sections where many sequences of approximating vector-valued stochastic processes are pushed through the It\^o and It\^o-Wentzell formula. The generalisation to different dimensions is straightforward from our text.

\subsection{The It\^o-Wentzell formula (classic)}

We first introduce the stochastic process $(X_t)_{t \in [0,T]}$ satisfying the dynamics 
\begin{align}
\label{eq:ProcessX-fullFlowMeasures}
\dd X_t = \beta_t\dd t +\gamma_t \dd W_t,\quad \textrm{and initial condition }X_0, 
\end{align}
where $W$ is a $d$-dimensional Brownian motion. The involved parameters satisfy the next condition.
\begin{assumption}
\label{Assump:SDE-X-1BM}
Let $X_0 \in L^2(\Omega,\cF_0,\bP;\bR)$ ($X_0$ is $\cF_0$-measurable and independent of $W_t,~ t \in [0,T]$). Take $\beta:\Omega\times [0,T]\to \bR^d$ and $\gamma:\Omega\times [0,T]\to\bR^{d \times d}$ such that $(\beta_t)_{t \in [0,T]}, (\gamma_t)_{t \in [0,T]}$ are $\bF$-progressively measurable processes and satisfy
\begin{align*}
\int_0^T (|\beta_s| + |\gamma_s|^2)\dd s < \infty, \quad \bP\textrm{-a.s..}
\end{align*}
\end{assumption}
We recall the It\^o-Wentzell formula in the style of \cite{cardaliaguet2019master}*{Section A.3.1} (see also \cite{kunita1997stochastic}*{Theorem 3.3.1} or \cite{nicole2013exact}*{Theorem 1.4}). 
\begin{theorem}[It\^o-Wentzell]
\label{theo-Classic-Ito-Wentzell}
Take $(X_t)_{t \in [0,T]}$ given by \eqref{eq:ProcessX-fullFlowMeasures} under Assumption \ref{Assump:SDE-X-1BM}. Let a map $V:\Omega \times [0,T] \times \bR^d\mapsto \bR$ be such that:
\begin{enumerate}[i)]
\item Fix $x \in \bR^d,$ $( V_t(x) )_{t \in [0,T]}$ is a continuous adapted process taking values in $\bR$;
\item Fix $t \in [0,T], \:\omega \in \Omega,\: \bR^d \ni x\mapsto V_t(x)$ is a $\cC^2(\bR^d)$-mapping with values in $\bR$;
\item $( V_t(x) )_{t \in [0,T]},$  $x \in \bR^d$ is a random field that admits the It\^o dynamics
\begin{align*}
V_t(x)= V_0(x) + \int_0^t \phi_s(x) \dd s + \int_0^t \psi_s(x)\cdot \dd W_s , \quad t \in [0,T],
\end{align*}
where $( \phi_t(\cdot) )_{t \in [0,T]}$ and $( \psi_t(\cdot) )_{t \in [0,T]}$ are $\bF$-progressively measurable processes with values in $\cC^2(\bR^d,\bR)$ and $\cC^2(\bR^d,\bR^{d})$ respectively, such that for any compact $K \subset \bR^d $
\begin{align*}
\int_0^T \big(\, \|\phi_s(\cdot) \|_{\cC^1(K)}  + \| \psi_s(\cdot) \|^2_{\cC^2(K)} \, \big) \dd s < \infty \quad \bP\textrm{-a.s.}
\end{align*}
\end{enumerate}
Then $( V_t(X_t) )_{t \in [0,T]}$ is an It\^o process and it satisfies $\bP$-a.s. the following expansion
\begin{align*}
V_T(X_T) 
-
V_0(X_0)&= \int_0^T\phi_s(X_s)\dd s 
+ \int_0^T\psi_s(X_s) \cdot \dd W_s 
+ \int_0^T \partial_x V_s(X_s) \cdot\beta_s \dd s \\
&
\qquad + \int_0^T\partial_x V_s(X_s) \cdot \gamma_s \dd W_s
+ \int_0^T\frac{1}{2}~\trace\big\{\partial_{xx}^2V_s(X_s) ~\gamma_s(\gamma_s)^{\intercal} \big\} \dd s \\
&
\qquad + \int_0^T \trace\big\{\partial_x\psi_s(X_s)  (\gamma_s)^\intercal \big\} \dd s.
\end{align*}

\end{theorem}
The first two terms correspond to dynamics of the field $V_t(\cdot)$ within installed $X_t$-trajectories. The next three terms correspond to the usual It\^o formula. The last term is a cross-variation of the diffusion factor of the process with the same nature noise induced by the stochastic field $V_t(\cdot)$ which we write using a matrix-trace notation, this is a short notation to describe the sum over $i \in 1,\dots, N$ of the cross variations $\langle \int_0^\cdot \partial_x \psi_{s,i}(X_s) \cdot \dd W_s , \int_0^\cdot \gamma_{s,i} \cdot \dd W_s\rangle_t$, where $\gamma_{\cdot,i}$ stands for the $i$-th row of $\gamma$ and $\partial_x \psi_{\cdot,i}(\cdot)$ stands for the gradient (in $x$) of the $i$-th entry of $\psi$.
\begin{proof}
In this formulation, we state conditions on the differentiability of $\phi,\psi$ directly as opposed to the original formulation by \cite{kunita1997stochastic}*{Theorem 3.3.1} where conditions over the characteristics of the driving semimartingale were given, \cite{kunita1997stochastic}*{Exercise 3.1.5} closes the gap.
\end{proof}
A close inspection of Theorem \ref{theo-Classic-Ito-Wentzell} and its proof (\cite{kunita1981some},\cite{kunita1997stochastic}) reveals that the theorem holds under reduced regularity requirements. We explore this observation with our next result.
\begin{theorem}[It\^o-Wentzell under reduced regularity]
\label{theo:ExtendItoWentzell}
The conclusion of Theorem \ref{theo-Classic-Ito-Wentzell} still holds for $( V_t(X_t) )_{t \in [0,T]}$ if in condition iii) the constraints on $\phi, \psi$ are replaced by:
\begin{itemize}
    \item[] $( \phi_t(\cdot) )_{t \in [0,T]}$ are $( \psi_t(\cdot) )_{t \in [0,T]}$ $\bF$-progressively measurable processes with values on the spaces $\cC^0(\bR^d,\bR)$ and $\cC^1(\bR^d,\bR^{d})$ respectively, such that for any compact $K \subset \bR^d $
\begin{align}
\label{eq:RedudeIntegrabilityClassicItoWentzell}
\int_0^T \big(\, \|\phi_s(\cdot) \|_{\cC^0(K)}  + \| \psi_s(\cdot) \|^2_{\cC^1(K)} \, \big) \dd s < \infty \quad \bP\textrm{-a.s.}
\end{align}
\end{itemize}
\end{theorem}
\begin{proof}
The arguments we use are classical. We mollify $V,\phi,\psi$ in their spatial components by convolution with a smoothing kernel and obtain a sequence $(V^n,\phi^n,\psi^n)$, $n\in\bN$, such that for each $n\in\bN$ $( \phi^n_t(\cdot) )_{t \in [0,T]}$ are $( \psi^n_t(\cdot) )_{t \in [0,T]}$ $\bF$-progressively measurable processes with values in $\cC^1(\bR^d,\bR)$ and $\cC^2(\bR^d,\bR^{d})$ respectively (in fact even more due to the mollification), such that for any compact $K \subset \bR^d$, $\bP\textrm{-a.s.}$
\begin{align}
\label{eq:RedudeIntegrabilityClassicItoWentzell-Sequence}
\int_0^T \big(\, \|\phi^n_s(\cdot) \|_{\cC^1(K)}  
+ \| \psi^n_s(\cdot) \|^2_{\cC^2(K)} \, \big) \dd s + 
\sup_{\hat n\in\bN} \int_0^T \big(\, \|\phi^{\hat n}_s(\cdot) \|_{\cC^0(K)}
+ \| \psi^{\hat n}_s(\cdot) \|^2_{\cC^1(K)} \, \big) \dd s < \infty.
\end{align}
Lastly, $\bP$-a.s.~for $t\in[0,T]$ a.e.~we have that $\phi^n_t, \psi^n_t,\partial_x \psi^n_t$ converge to $\phi_t, \psi_t,\partial_x \psi_t$ uniformly (in $n$) on compact sets.  
It is clear that $V^n$ retains the properties of $V$, uniformly over $n$ for the $0$-th, $1$-st and $2$-nd derivative. In particular, $\bP$-a.s. for any $t\in[0,T]$  $V^n_t,\partial_{x}V^n_t,\partial_{xx}^2V^n_t$ converge to $V_t,\partial_{x}V_t,\partial_{xx}^2 V_t$ uniformly on compact sets. We conclude via Theorem \ref{theo-Classic-Ito-Wentzell} that $( V^n_t(X_t) )_{t \in [0,T]}$ is an It\^o process satisfying the expansion given.

The passage to the limit as $n\to \infty$ is also argued in a classical way. First we make use of a localizing sequence $(\tau^m)_{m\in\bN}$ over $X$ defined as $\tau^m:=\inf\{t>0:|X_t|>m\}$, $m\in\bN$ which in turn allows us to make use of the uniform convergence over compacts for the maps' sequence (in $n$) and \eqref{eq:RedudeIntegrabilityClassicItoWentzell}-\eqref{eq:RedudeIntegrabilityClassicItoWentzell-Sequence} repeatedly, i.e.~we can assume that $X$ is bounded. Arguing convergence of the Lebesgue integrals follows via continuity of the maps, integrability of the coefficients (see Assumption \ref{eq:ProcessX-fullFlowMeasures}) and dominated convergence theorem taking advantage of uniform convergence over compacts given that $X$ is assumed to take values in a bounded set. The stochastic integral terms requires an additional argument which we provide for the 2nd integral (the 1st is handled similarly),
\begin{align*}
\bE\Big[\sup_{0\leq t\leq T} 
&
\big|\int_0^t\partial_x V^n_s(X_s) \cdot \gamma_s \dd W_s
-\int_0^t\partial_x V_s(X_s) \cdot \gamma_s \dd W_s
\big|^2\Big]
\\
&
\leq \bE\Big[ \int_0^T \big|\partial_x V^n_s(X_s)- \partial_x V_s(X_s)\big|^2 |\gamma_s|^2 \dd  s
\Big].
\end{align*}
Since $\partial_x V,\partial_x V^n$ are jointly continuous in their variables and converge uniformly over compacts, $X$ is assumed to take values in a bounded set and $\gamma$ satisfies Assumption \ref{Assump:SDE-X-1BM}, then the RHS converges to zero as $n\to \infty$.
\end{proof}

\subsection{The Lions derivative}

\subsubsection{The Lions derivative and notational conventions}
\label{sec:LionsDerivative}

To consider the calculus for the mean-field setting one requires to build a suitable differentiation operator on the $2$-Wasserstein space. Among the several notions of differentiability of a functional $u$ defined over  $\cP_2(\bR^d)$ we follow the approach introduced by Lions in his lectures at Coll\`ege de France \cite{lions2007cours} and further developed in \cite{cardaliaguet2010notes}. A comprehensive presentation can be found in the joint monograph of Carmona and Delarue \cite{CarmonaDelarue2017book1},\cite{CarmonaDelarue2017book2}.

\begin{remark}[The intrinsic and the Lions derivative]
\label{rem:On Interincsics}
    We follow the measure derivative approach by Lions. We point that this notion of derivative is a stronger notion of derivative than the intrinsic measure derivative concept introduced in \cite{AlbeverioKondratievRockner1996} (see also the Appendix in \cite{RenRocknerWang2022linearization}). See \cite{RenWang2021DerivativeFormulas} for further details and characterisations on the relations of different types of derivatives in measures. 
\end{remark}
We consider a canonical lifting of the function $u:\cP_2(\bR^d) \to \bR$ to $\tilde{u}: L^2(\Omega,\cF,\bP;\bR^d) \ni X \to \tilde u (X) = u(Law(X)) \in \bR$, where  $L^2(\Omega,\cF,\bP;\bR^d)$ is a space of square integrable random variables. We can say that $u$ is $L$-differentiable at $\mu$, if $\tilde u$ is Frech\`et differentiable (in $L^2$) at some $X$, such that $\mu = \bP \circ X^{(-1)}$. Denoting the gradient by $D\tilde u$ and using a Hilbert structure of the $L^2$ space, we can identify $D\tilde u$ as an element its dual, $L^2$ itself. It was shown in \cite{cardaliaguet2010notes} that $D\tilde u$ is a $\sigma(X)$-measurable random variable and given by the function $Du(\mu, \cdot) : \bR^d \to \bR^d $, depending on the law of $X$ and satisfying $Du (\mu, \cdot) \in L^2(\bR^d, \cB(\bR^d),\mu; \bR^d)$. Hereinafter the  $L$-derivative of $u$ at $\mu$ is the map $\partial_\mu u(\mu,\cdot): \bR^d \ni v \to \partial_\mu u(\mu,v) \in \bR^d$, satisfying $D\tilde u(X) = \partial_\mu u(\mu,X)$.
We always denote $\partial_\mu u$ as the version of the $L$-derivative that is continuous in the product topology of all components of $u$. Moreover, let $\partial^2_{\mu}$ denote second derivative in measure and $\partial_v \partial_\mu u$ denote the derivative with respect to new variable arisen after applying derivative in measure. The notion of $\partial_\mu^2$ is chosen in favour of $\partial_{\mu\mu}^2$, as the latter may be hinting at the linear nature of $L$-derivative, that is not the case at all.

When we do the lift $\tilde{\xi}$ and $\hat{\xi}$ are the lifted random variables defined over the twin stochastic spaces $(\tilde{\Omega}, \tilde{\cF},\tilde{\bP})$ and $(\hat{\Omega}, \hat{\cF}, \hat{\bP})$ respectively, having the same law $\mu$.
We form a new probability space $(\Omega,\cF,\bP) \times(\tilde{\Omega}, \tilde{\cF}, \tilde{\bP})$ and consider random variables $\tilde{\xi}(\omega,\tilde{\omega}) = \xi(\tilde{\omega})$. Since this procedure is valid for the stochastic processes on respective stochastic bases $(\tilde{\Omega}, \tilde{\cF}, \tilde{\bF}=\tilde{(\cF_t)}_{t \in [0,T]},\tilde{\bP})$ and $(\hat{\Omega}, \hat{\cF}, \hat{\bF} = \hat{(\cF_t)}_{t \in [0,T]},\hat{\bP})$, one can consider $(X_t,\tilde{X}_t,\hat{X}_t)$ as a triple of independent identically distributed processes. The same applies to a finite amount of copy spaces $\:(\Omega^{l}, \cF^{l}, \bF^{l} = (\cF^{l}_t)_{t \in [0,T]},\bP^{l}),1 \leq l \leq N \in \bN$ to form a new product space and the respective tuple $(X_t,\tilde{X}_t,\hat{X}_t,X^{1}_t, \dots, X^{N}_t)$ remains mutually independent.

We will add the bases $(\tilde{\Omega}, \tilde \cF,\tilde{\bF}, \tilde{\bP})$ and $(\hat{\Omega},\hat \cF,\hat{\bF},\hat{\bP})$ and further use them as an environment for model representatives of the mean-field (each living in the distinct respective space), whereas sampling from the mean-field will give us $N$ particles living within respective spaces $(\Omega^{l}, \cF^{l}, \bF^{l} = (\cF^{l}_t)_{t \in [0,T]},\bP^{l}), \: 1\leq l \leq N,$ to be used within the propagation of chaos procedures below.  Hereinafter $\tilde{\bE}$ denotes the expectation acting on the model twin space $\tilde{\Omega}$.

Over the present work we omit the re-notation after adding some new probability spaces, but will assume that adding a copy processes automatically intimates the procedure described above. The common noise setting given in Section 4 requires a slightly variation of this approach which we disclose in the proof of Theorem \ref{theo:IW-PartialFlow-measureOnly}.

\subsubsection{Regularity in the measure argument}
In this section we recall several spaces of measure-regularity arising in the literature on Wasserstein calculus.
\begin{definition}
\label{def:FullC2}
We say the functional $u:\cP_2(\bR^d)\to \bR$ is Fully $\cC^{2}(\cP_2(\bR^d))$ if

\begin{enumerate}[i)]
	\item $u$ is $L$-differentiable at every point $\mu \in \cP_2(\bR^d)$, and $\partial_\mu u: \cP_2(\bR^d) \times \bR^d \to \bR$ has a $\mu$-version such that $\partial_\mu u(\mu,v)$ is joint-continuous at every pair $(\mu,v) \in \cP_2(\bR^d) \times \bR^d$;
	
	\item For any $\mu \in \cP_2(\bR^d)$, the map $v \mapsto \partial_\mu u(\mu,v)\in \bR^d$ is $\bR^d$-differentiable at every point $v\in \bR^d$; and $\partial_v\partial_\mu u: \cP_2(\bR^d) \times \bR^d \to \bR^{d \times d}$ has a $\mu$-version such that $\partial_v\partial_\mu u(\mu,v)$ is joint-continuous at every pair $(\mu,v) \in \cP_2(\bR^d) \times \bR^d$;
	
	\item For any $v \in \bR^d$, the map $\mu \mapsto \partial_\mu u(\mu,v): \in \bR^d$ is $L$-differentiable at every point $ \mu\in \bR^d$, and $\partial^2_{\mu} u: \cP_2(\bR^d) \times \bR^d \times \bR^d \to \bR^{d\times d}$ has a $\mu$-version such that $\partial_\mu^2 u(\mu,v,v^\prime)$ is joint-continuous at every triple $(\mu,v,v^\prime) \in \cP_2(\bR^d) \times \bR^d \times \bR^d.$
\end{enumerate}
\end{definition}

We next restrict the regularity with respect to the space variable arising after taking measure derivative to the $\Supp(\mu)$, since in our probabilistic setting the process sitting there obviously will not escape this set. This restriction comes from the  interplay with the \emph{Partial}-$\cC^2$-regularity of  \cite{CarmonaDelarue2017book1}*{Chapter 5.6.4}.
\begin{definition}
\label{def:PartiallyC2}
We say the function $u:\cP_2(\bR^d)\to \bR$ is Partially $\cC^{2}(\cP_2(\bR))$ if
\begin{enumerate}[i)]
	\item $u$ is $L$-differentiable at every point $\mu \in \cP_2(\bR^d)$, such that $\partial_\mu u$ has a $\mu$-version that is locally bounded and joint-continuous at every pair $(\mu,v), \mu \in \cP_2(\bR^d),~ v \in\Supp(\mu)$;
	\item For any $v \in \bR^d$, the map $\bR^d \ni v \mapsto \partial_\mu u(\mu,v) \in \bR$ is $\bR^d$-differentiable at every point $ v\in \Supp(\mu)$. Moreover, $\partial_v\partial_\mu u: \cP_2(\bR^d) \times \bR^d \to \bR^d \otimes \bR^d$ has a $\mu$-version such that $\partial_v\partial_\mu u(\mu,v)$ is locally bounded and joint-continuous at every pair $(\mu,v),~ \mu \in \cP_2(\bR^d),~ v \in\Supp(\mu)$.
\end{enumerate}
\end{definition}
This regularity level does not require a second Frech\'et derivative of the lift to exist. Looking ahead, we do not expect to receive any second-order terms in the expansion of the measure component, hence it is quite essential not to demand such a regularity (see Theorem \ref{theo:IW-FUllmeasureFlow-Proc+Measure-1BM} or 
Theorem \ref{theo:IW-FullFlow-measureOnly} below).

For the purpose of Theorem \ref{theo-Classic-Ito-Lions} we require the regularity in all components, and we introduce the following definition.
\begin{definition}
\label{def:C1211}
A function $u:[0,T] \times \bR^d \times \cP_2(\bR^d)\to \bR$ is $\cC^{1,2,(1,1)}$ if

\begin{enumerate}[i)]
    \item For any $\mu\in\cP_2(\bR^d)$ the map $[0,T]\times \bR^d \ni (t,x)\mapsto u_t(x,\mu)$ is $\cC^{1,2}$, and the maps $\partial_t u$, $\partial_x u$ and $\partial_{xx}^2 u$ are joint-continuous at every triple $ (t,x,\mu) \in [0,T] \times \bR^d \times \cP_2(\bR^d)$;

	\item For any $(t,x) \in [0,T]\times \bR^d$, the map $\mu \mapsto u_t(x,\mu)$ is continuously L-differentiable at every point $\mu \in \cP_2(\bR^d)$. Moreover, $\partial_\mu u: [0,T] \times \bR^d \times \cP_2(\bR^d) \times \bR^d \to \bR^d$ has a $\mu$-version such that $\partial_\mu u_t(\mu,v)$ is joint-continuous and locally bounded at every quadruple $(t,x,\mu,v)$, with $(t,x,\mu) \in [0,T] \times \bR^d \times \cP_2(\bR^d),~ v \in \Supp(\mu)$;
	
	\item For any $(t,x,\mu)\in [0,T] \times \bR^d \times \cP_2(\bR^d)$, the map $v \mapsto \partial_\mu u_t(x,\mu,v)$ is continuously $\bR^d$-differentiable at every point $v \in \bR^d$. Moreover, its derivative $\partial_v \partial_\mu u: [0,T] \times \bR^d \times \cP_2(\bR^d) \times \bR^d \to \bR^{d\times d}$ has a $\mu$-version such that $\partial_v\partial_\mu u_t(\mu,v)$ is continuous and locally bounded at every quadruple $(t,x,\mu,v)$, with $(t,x,\mu) \in [0,T] \times \bR^d \times \cP_2(\bR^d),~ v \in \Supp(\mu)$.
\end{enumerate}
\end{definition}

\subsubsection{The Empirical projection map}
We recall the concept of \emph{empirical projection map} given in \cite{chassagneux2014classical} which will be one of the main workhorses throughout our work.
\begin{definition}[Empirical projection of a map]
\label{def:Auxiliary-uN-for-EmpirialTrick}
Given $u: \cP_2(\bR^d) \to \bR$ and $N\in \bN$, define the empirical projection $u^N$ of $u$ via $u^N: (\bR^d)^N \to \bR$, such that $$u^N(x^1,\dots, x^N) := u \big(\bar{\mu}^N\big),
\quad \text{with}\quad
\bar{\mu}^N := \frac{1}{N}\sum\limits_{l=1}^N \delta_{x^l}
\quad\textrm{and}\quad 
x^l\in \bR^d ,~ l=1,\dots,N.
$$
\end{definition}
We recall \cite{CarmonaDelarue2017book1}*{Proposition 5.91 and Proposition 5.35} which relates the spatial derivative of $u^N$ with the $L$-derivative of $u$.
\begin{proposition}
\label{prop:DerivativeRelations-Space-2-Lions}
Let $u: \cP_2(\bR^d) \to \bR$ be \textit{Fully}-$\cC^2 (\cP_2(\bR^d))$, then, for any $N>1$, the empirical projection $u^N$ is $\cC^2$ on $(\bR^d)^N$ and for all $x^1,\cdots,x^N\in\bR^d$ we have the following differentiation rules
\begin{align*}
    \partial_{x^j}u^N(x^1, \dots, x^N) &= \frac{1}{N}\: \partial_\mu u\Big(\frac{1}{N}\sum_{l=1}^N \delta_{x^l},x^j\Big),
    \\
    \partial_{x^k}\partial_{x^j}u^N(x^1, \dots, x^N) 
    &
    = \frac{1}{N}\: \partial_v\partial_\mu u\Big(\frac{1}{N}\sum_{l=1}^N \delta_{x^l},x^j\Big) \1_{j=k} + \frac{1}{N^2} \:\partial^2_{\mu} u\Big(\frac{1}{N}\sum_{l=1}^N \delta_{x^l},x^j,x^k\Big).
\end{align*}
\end{proposition}

\subsection{It\^o-Lions chain rule along a full flow of measures (classic)}

Alongside $(X_t)_{t \in [0,T]}$ given by \eqref{eq:ProcessX-fullFlowMeasures} we introduce another process $(Y_t)_{t \in [0,T]}$ and its law $(\mu_t)_{t\in[0,T]}$. Take $W$ as a  $d$-dimensional Brownian motion and let $(Y_t)_{t \in [0,T]}$ satisfy the dynamics 
\begin{align}
\label{eq:GenericYprocess-FullFlow22}
\dd Y_t = b_t\dd t + \sigma_t \dd W_t,\quad \textrm{and initial condition }Y_0, 
\end{align}
where we denote the law of $Y_t$ by $\mu_t:=\bP\circ Y^{(-1)}_t,~t\in[0,T]$ and the associated coefficients satisfy the below assumption.
\begin{assumption}
\label{Assump:SDE-Y-mu-1BM}
    Let $Y_0 \in L^2(\Omega,\cF_0,\bP)$ ($Y_0$ is $\cF_0$-measurable and independent of $W_t, ~ t \in [0,T]$). Take $b:\Omega\times [0,T]\to \bR^d$ and $\sigma:\Omega\times [0,T]\to\bR^{d \times d}$ such that $(b_t)_{t \in [0,T]}, (\sigma_t)_{t \in [0,T]}$ are $\bF$-progressively measurable processes and satisfy
    \begin{align*}
        \bE\big[ \int_0^T |b_s|^2 + |\sigma_s|^4 \dd s \big] < \infty.
    \end{align*}
\end{assumption}

The requirements of higher integrability of the involved coefficients stems from the proof methodology we implement. Namely, the convergence of the formula for the mollified version (see Step 1 of the Proof of \ref{theo:IW-FullFlow-measureOnly}) follows from these higher moment bounds (for more details see \cite[Proof of Lemma 5.95]{CarmonaDelarue2017book1}). At the same time we emphasise that one can reduce the integrability of $b$ and $\sigma$ at the expense of asking for higher moments for measure derivative terms.

\begin{remark}
One can take ``closed-loop'' type dependence for the coefficients, i.e.~coefficients of the form $\hat{b}_t := b_t(Y_t,\mu_t)$ and $ \hat{\sigma}_t:=\sigma_t(Y_t,\mu_t)$, since our setting covers all the special cases. In fact, an existence \& uniqueness result for the SDE for $Y$ allows to freeze the components inside the coefficients and with sufficient integrability the ``frozen'' SDE follows the dynamics \eqref{eq:GenericYprocess-FullFlow22}. 
\end{remark}
For completeness we recall the It\^o-Lions formula \cite{CarmonaDelarue2017book1}*{Proposition 5.102} for deterministic maps (see also \cite{Wang2021Imagedependent}) following the framework Section \ref{sec:LionsDerivative}, recall that $\tilde{\bE}$ denotes the expectation acting on the model twin space $(\tilde{\Omega}, \tilde{\bF}, \tilde{\bP})$ and let the processes $(\tilde{Y}_t, \tilde{b}_t, \tilde{\sigma}_t)_{t\in[0,T]}$ be the twin processes of $(Y_t,b_t,\sigma_t)_{t\in[0,T]}$ respectively living within.

\begin{theorem}
\label{theo-Classic-Ito-Lions}
Let $u:[0,T] \times \bR^d \times \cP_2(\bR^d) \to \bR$ be $\cC^{1,2,(1,1)}$. Furthermore, for any compact $K \subset \bR^d \times \cP_2(\bR^d)$ we have
\begin{align*}
	\sup_{(t,x,\mu) \in [0,T]\times K}\Big\{\int_{\bR^d}\Big[|\partial_\mu u_t(x,\mu, v)|^2 + |\partial_v\partial_\mu u_t(x,\mu, v)|^2\Big]\mu(\dd v)\Big\}  < \infty.
\end{align*}
Take $(X_t)_{t \in [0,T]}$ given by \eqref{eq:ProcessX-fullFlowMeasures} under Assumption \ref{Assump:SDE-X-1BM} and take $\mu$ associated to \eqref{eq:GenericYprocess-FullFlow22} under Assumption \ref{Assump:SDE-Y-mu-1BM}. Then $(u_t(X_t,\mu_t))_{t\in[0,T]}$ is an It\^o process satisfying $\bP$-a.s.

\begin{align*}
u_T(X_T,\mu_T)  -&  u_0(X_0,\mu_0)\\
=&  
\int_0^T \partial_t u_s(X_s,\mu_s) \dd s  
+ \int_0^T \Big[ \partial_xu_s(X_s,\mu_s) \cdot\beta_s  \dd s
+ \int_0^T \partial_x u_s(X_s,\mu_s) \cdot \gamma_s \dd W_s\\
&
+ \int_0^T\frac{1}{2}~\trace \big\{\partial_{xx}^2 u_s(X_s,\mu_s)  ~\gamma_s (\gamma_s)^\intercal  \big\} \dd s
 \\
&
+ \int_0^T \tilde{\bE} \Big[\partial_\mu u_s(X_s,\mu_s,\tilde{Y}_s)\cdot \tilde{b}_s\Big]\dd s
+ \int_{0}^{T}\frac{1}{2}~\tilde{\bE}\Big[\trace \big\{\partial_v\partial_\mu u_s(X_s,\mu_s, \tilde{Y}_s)~\tilde{\sigma}_s(\tilde{\sigma}_s)^{\intercal}\big\}\Big] \dd s.
\end{align*}
\end{theorem}


%
%
%
%

\section{It\^o-Wentzell-Lions chain rule with a full flow of measures}
\label{sec:fullflow}

As it was shown in \cite{chassagneux2014classical}, one can apply an approach based on empirical projections to built the chain rule. This approach very convenient since with it we are able  require (loosely) the same regularity as in Theorem \ref{theo-Classic-Ito-Lions} above. One can notice that the  second measure derivative term of the formulae appearing within measure argument expansion vanishes when applying the limit procedure. Nonetheless, in order to argue via Taylor expansions the second derivative in measure has to exist which is a very strong assumptions. We can avoid this requirement using this technique.

Let $u : \Omega \times [0,T] \times\bR^d\times \cP_2(\bR^d)\to \bR$ be a random field, satisfying the expansion 
\begin{align}
\label{eq:RandomField-1BM}
\dd u_t(x,\mu) = \phi_t(x,\mu)\dd t+\psi_t(x,\mu)\cdot \dd W_t,\quad u_0(x,\mu) = f(x,\mu),
\end{align}
where $f(x,\mu): \bR^d \times \cP_2(\bR^d) \to \bR$ is a deterministic function, $(W_t)_{t \in [0,T]}$ is a $d$-dimensional $\bF$-Brownian motion, $(\phi, \psi): \Omega\times [0,T] \times \bR^d \times \cP_2(\bR^d) \to \bR \times \bR^{d}$ are $\bF$-progressively measurable processes. 

Throughout we will work with the law $(\mu_t)_{t\in[0,T]}$ of the process  $(Y_t)_{t\in[0,T]}$ given in \eqref{eq:GenericYprocess-FullFlow22} under Assumption \ref{Assump:SDE-Y-mu-1BM}. In the second portion of the section, we additionally work with $(X_t)_{t\in[0,T]}$ solution to \eqref{eq:ProcessX-fullFlowMeasures} under under Assumption \ref{Assump:SDE-X-1BM}.

\subsection{It\^o-Wentzell-Lions formula for measure functionals}
We start by discussing the measurability of the involved structures and for which the following remark addresses the issue for the whole manuscript.
\begin{remark}[On measurability]
 \label{rem:measurability-main-remark}
 The measurability of the measure expansion component is deeply discussed in \cite{CarmonaDelarue2017book1}*{Remarks 5.101 and 5.103}. Within the present work we are interested in conditioning on the field noise, the matter of which is discussed in \cite{CarmonaDelarue2017book2}*{Section 4.3}. We refer the reader to this monograph for comprehensive and detailed approach.
\end{remark}

\subsubsection{It\^o-Wentzell expansion}
In this subsection we work with the It\^o random field \eqref{eq:RandomField-1BM} and we keep $x \in \bR^d$ at some fixed value for the whole subsection and hereinafter we will omit its presence within $u, \phi$ and $\psi$, i.e.~we set
\begin{align*}
(t,x,\mu)\in[0,T]\times \bR^d\times \cP_2(\bR^d) \ \ 
u_t(\mu):=u_t(x,\mu), \quad
\phi_t(\mu):=\phi_t(x,\mu), 
\quad\textrm{and}\quad
\psi_t(\mu):=\psi_t(x,\mu). 
\end{align*}
Similarly to the full- and partial-$\cC^2$ maps concept in Definition \ref{def:FullC2} and \ref{def:PartiallyC2}, we introduce the concept of a \emph{partially-$\cC^{2}$ It\^o random field}, describing the field's regularity in the measure component and we coin it  \emph{RF-Partially $\cC^{2}$}.
\begin{definition}
\label{def:PartialC2-RandomField-FullFlow}
We say the random field $u:\Omega\times [0,T] \times \cP_2(\bR^d)\to \bR$ given in \eqref{eq:RandomField-1BM} (for some $x\in \bR^d$ fixed) is \emph{RF-Partially}-$\cC^{2}$ if 
\begin{enumerate}[i)]
    \item For any $\mu \in \cP_2(\bR^d)$,  $(u_t(\mu))_{t\in[0,T]}$ is a continuous adapted process taking values over $\bR^d$ and $(\phi_t(\mu))_{t \in [0,T]},(\psi_t(\mu))_{t \in [0,T]}$ are $\bF$-progressively measurable processes with values in $\bR$ and $\bR^{d}$ respectively;
    
    \item For almost all $t \in[0,T]$, the maps $\mu\mapsto \phi_t(\mu),~ \mu\mapsto \psi_t(\mu)$ are $\bP$-a.s.~continuous in the topology induced by the  Wasserstein metric for any $\mu \in \cP_2(\bR^d)$;
    
	\item For any $t\in[0,T]$ the map $\mu\mapsto u_t(\mu)$ is $\bP$-a.s.~continuous in topology, induced by Wasserstein metric and $L$-differentiable $\bP$-a.s. at every $\mu \in \cP_2(\bR^d).$ Moreover, $\partial_\mu u_t(\mu,v)$ has a $\mu$-version such that $\partial_\mu u_t(\mu,v)$ is $\bP$-a.s.~joint-continuous at every triple $(t,\mu,v)$ with $(t,\mu) \in [0,T] \times \cP_2(\bR^d),~v \in \Supp(\mu), ~\bP$-a.s.;
	
	\item For any $(t,\mu) \in [0,T]\times \cP_2(\bR^d)$ the map $v \mapsto \partial_\mu u_t(\mu,v)$ is $\bR^d$-differentiable $\bP$-a.s.~at every $ v\in \Supp(\mu)$. Moreover, the map $\partial_v\partial_\mu u_t(\mu,v)$ has a $\mu$-version such that $\partial_v\partial_\mu u_t(\mu,v)$ is $\bP$-a.s.~joint-continuous at every triple $(t,\mu,v)$, with $(t,\mu) \in [0,T] \times \cP_2(\bR^d),~v \in \Supp(\mu), ~\bP$-a.s..
\end{enumerate}
\end{definition}

\begin{remark}
In contrast with \cite{CarmonaDelarue2017book1,CarmonaDelarue2017book2}, where the local boundedness condition is present in the regularity conditions, we restrict ourselves to the continuous version of the Lions derivative from the beginning, hence local boundedness is automatically implied by the continuity. 
\end{remark}

The main proof mechanics relies on the projection over empirical distributions technique as explored in \cites{chassagneux2014classical,CarmonaDelarue2017book1}.  Recall  that $\tilde{\bE}$ denotes the expectation acting on the model twin space $(\tilde{\Omega}, \tilde{\bF}, \tilde{\bP})$ and let the processes $(\tilde{Y}_t, \tilde{b}_t, \tilde{\sigma}_t)_{t\in[0,T]}$ be the twin processes of $(Y_t,b_t,\sigma_t)_{t\in[0,T]}$ respectively living within (see Section \ref{sec:LionsDerivative}).

\begin{theorem} 
\label{theo:IW-FullFlow-measureOnly}
Let $u$ be the \emph{RF-Partially}-$\cC^{2}$ It\^o random field \eqref{eq:RandomField-1BM} (where $x\in \bR^d$ is fixed and omitted throughout, also for $\phi$ and $\psi$). Assume for any compact $K\subset \cP_2(\bR^d)$ and for any $0 \leq t<T$ that 
\begin{align*}
	\int_0^t \sup_{\mu \in K}\Big\{|\phi_s(\mu)| + |\psi_s(\mu)|^2 \Big\}\dd s  < \infty,\quad \bP \textit{-a.s.},
\end{align*}
and
\begin{align}
\label{cond:MeasureOnly-1BM-u}
 \sup_{(t,\mu) \in [0,T]\times K}\Big\{\int_{\bR^d}\Big[|\partial_\mu u_t(\mu, v)|^2 + |\partial_v\partial_\mu u_t(\mu, v)|^2\Big]\mu(\dd v)\Big\}  < \infty,\quad \bP \textit{-a.s..}
\end{align}

Let $(\mu_t)_{t \in [0,T]}$ be the law of the solution to \eqref{eq:GenericYprocess-FullFlow22} satisfying Assumption \ref{Assump:SDE-Y-mu-1BM}. 
Then $(u_t(\mu_t))_{t\in[0,T]}$ is an It\^o process $\bP$-a.s. satisfying the expansion
\begin{align*}
u_T(\mu_T) - u_0(\mu_0) &=\int_0^T\phi_s(\mu_s) \dd s + \int_0^T\psi_s(\mu_s) \cdot \dd W_s
+
\int_0^T \tilde{\bE} \Big[ \partial_\mu u_s(\mu_s,\tilde{Y}_s) \cdot \tilde{b}_s \Big]\dd s \\
&
\qquad \qquad + 
 \int_{0}^{T} \frac{1}{2} ~\tilde{\bE} \Big[  \trace \big\{\partial_v\partial_\mu u_s(\mu_s, \tilde{Y}_s) 
~\tilde{\sigma_s}(\tilde{\sigma}_s)^{\intercal} \big\} \Big]\dd s.
\end{align*}
\end{theorem}

\begin{remark}
\label{rem:IgnoreItoWentzelAndJustApplyLions}
Following from Theorem \ref{theo-Classic-Ito-Lions} we have that for fixed $r\in[0,T],~t\mapsto u(r,\mu_t)$
 $\bP$-a.s. satisfies the expansion
\begin{align*}
u_r(\mu_T) - u_r(\mu_0) &=
\int_0^T \tilde{\bE} \Big[ \partial_\mu u_r(\mu_s,\tilde{Y}_s) \cdot \tilde{b}_s \Big]\dd s \\
&
\qquad \qquad + 
 \int_{0}^{T} \frac{1}{2} ~\tilde{\bE} \Big[  \trace \big\{\partial_v\partial_\mu u_r(\mu_s, \tilde{Y}_s) 
~\tilde{\sigma_s}(\tilde{\sigma}_s)^{\intercal} \big\} \Big]\dd s.
\end{align*}
\end{remark}

\begin{remark}
\label{rem:SlightDifference2CDLL2019}
We highlight the requirement of the square integrability on $\partial_\mu u$ and $\partial_v\partial_\mu u$ in \eqref{cond:MeasureOnly-1BM-u} which is not present in \cite{cardaliaguet2019master}*{Appendix A}. The requirement is necessary for the intermediary step of $W_2$-convergence of the empirical measure appearing in those terms.
\end{remark}

\begin{remark}
Here we write $\trace$ within last term assuming the symmetry of respective matrix holding $\bP$-a.s. for any $t \in [0,T]$. One can see that within the approximating procedure, i.e.~the distance between the Hessian of the mollified empirical projection and the $\partial_v\partial_\mu u$-term is controlled through the decreasing sequence $\varepsilon_N \searrow 0$, thus the symmetry follows by approximation. See \cite{CarmonaDelarue2017book1}*{Remark 5.98} for details.
\end{remark}
\begin{proof}[Proof of Theorem \ref{theo:IW-FullFlow-measureOnly}]
For this proof we follow as guideline the proof of Theorem 5.99 in \cite{CarmonaDelarue2017book1}. Let throughout $t\in[0,T]$.
Recall that $\bE^{1,\dots,N}$ denotes an expectation with respect to the product of sample twin spaces $(\Omega^{1}, \bF^{1}, \bP^{1})\times \dots \times (\Omega^{N}, \bF^{N}, \bP^{N})$. We again underline that we act on an atomless Polish space. 

\emph{Step 1: Mollification \& compactification.} If the desired expansion holds true for any $u$ - \textit{RF-Partially} $\cC^2$, bounded and uniformly continuous (in space and measure arguments), then the formula holds for $u$ satisfying the conditions of the theorem. This fact is straightforward by applying a two-step mollification procedure in the vein of \cite{CarmonaDelarue2017book1}*{Theorem 5.99} and which we introduce next.

Defining for any $ t \in [0,T]$ the $(u \star \rho)_t(\mu) := u_t(\mu \circ \rho^{-1})$ with $\rho: \bR^d \to \bR^d$ smooth function with compact support, the $\bP$-a.s. boundedness of $(u \star \rho)_t(\mu), \partial_\mu(u \star \rho)_t(\mu,v)$ and $\partial_v\partial_\mu(u \star \rho)_t(\mu,v)$ follows from $\bP$-a.s. local boundedness of $u_t, \partial_\mu u_t (\mu,v)$ and $\partial_v\partial_\mu u_t (\mu,v)$ respectively. We also notice that $\partial_\mu (u \star \rho)_t (\mu,v)$ and $\partial_v\partial_\mu (u \star \rho)_t (\mu,v)$ are $\bP$-a.s. joint-continuous in every triple $t \in [0,T], \mu \in \cP^2(\bR^d), v \in \Supp(\mu)$. In order to obtain continuity over the whole space we smooth out the distribution by convolution with a Gaussian density, i.e.~considering $\mu \mapsto (u \star \rho)(\mu * \phi_G)$ instead of $\mu \mapsto (u \star \rho)(\mu)$ with $\phi_G$ - density of standard $d$-dimensional Gaussian distribution $N(0, I_d)$ on $\bR^d$ and $(\mu * \phi_G)(x):=\int_{\bR^d} \phi_G(x-y) d\mu(y)$. Now the support of $\mu *\phi_G$ is the whole $\bR^d$ and $\partial_\mu(u \star \rho)$ and $\partial_v\partial_\mu(u \star \rho)$ are $\bP$-a.s. continuous at every triple $(t,\mu * \phi_G, v), t \in [0,T], v \in \bR^d $.

Now we introduce $\phi_{\varepsilon,G}$ - Gaussian densities $N(0,\varepsilon I_d)$. Letting  $\varepsilon \searrow 0$ one can see convergence of $\phi_{\varepsilon,G}$ to Dirac measure at $0$ for the $W_2$ distance and thus convergence of $\partial_\mu (u \star \rho)_t(\mu * \phi_{\varepsilon,G},v)$ and $\partial_v \partial_\mu (u \star \rho)_t(\mu * \phi_{\varepsilon,G},v)$ to $\partial_\mu (u \star \rho)_t(\mu,v)$ and $\partial_v \partial_\mu (u \star \rho)_t(\mu,v)$ respectively for any $t \in [0,T], v \in \Supp(\mu)$. Now picking $\rho_n$ in a way that $(\rho_n, \partial_z\rho_n, \partial_{zz}^2\rho_n)(z) \to (z
, I_d,0)$ as $n\to \infty$, we can conclude that $\partial_\mu (u \star \rho_n)_t(\mu,v)$ and $\partial_v \partial_\mu (u \star \rho_n)_t(\mu,v)$ converge to $\partial_\mu u_t (\mu,v)$ and $\partial_v \partial_\mu u_t (\mu,v)$ $\bP$-a.s.. One should notice that all the conditions in the theorem hold true while doing mollification. Thus we can assume that $u$ and its first and partial second order derivatives are $\bP$-a.s. {uniformly} bounded and uniformly continuous, and $Y$ is a bounded process.

Now we are to show the well-posedness of the mollification scheme, i.e. that chain rule applied to $u_n:= u \star \rho_n$ converges to the one for $u$.
It is straightforward to verify that $u_n$ satisfies $\bP$-a.s.~\eqref{cond:MeasureOnly-1BM-u} uniformly in $n \geq 1$. 
We apply the dominated convergence theorem twice to conclude the $\bP$-a.s.~convergence for all the terms but the stochastic integral. To handle the latter one additionally requires an argument across the quadratic variation as written in Theorem \ref{theo:ExtendItoWentzell} and localisation.
 
\emph{Step 2. Wellposedness and approximation.}  For a smooth compactly supported density $\rho$ on $\bR^d$ we define, for $n\in \bN$, the mollified version $u^{N,n}$ of $u^N$ (introduced in Definition \ref{def:Auxiliary-uN-for-EmpirialTrick}) for any $t \in [0,T]$, any $y^1,\cdots,y^N \in \bR^d$ by 
\begin{align*}
u^{N,n}_t(y^1, \dots,y^N) :&= n^{Nd} \int_{(\bR^d)^N} u_t^N(y^1-z^1,\dots,y^N-z^N) \prod\limits_{l=1}^{N} \rho(nz^l)\prod\limits_{l=1}^{N}\dd z^l,
\end{align*}
where $\rho$ is a smooth and compactly supported density. We define $\phi^{N,n}, \psi^{N,n}$, in the same way as $u^{N,n}$. One can notice that  $u_t^{N,n},\phi^{N,n}_t, \psi^{N,n}_t$ are maps in $\cC^2\big((\bR^d)^N\big)$ and thus all derivatives up to second order exist and are regular. Furthermore, to $u^{N,n}$ one can apply the standard It\^o-Wentzell formulae, since it satisfies all the conditions of Theorem \ref{theo-Classic-Ito-Wentzell} (verified below).

Now we describe the approximation procedure. From the properties of the Wasserstein metric for finitely supported measures with uniformly bounded second moments, we have $$W_2\Big(\frac{1}{N}\sum\limits_{i=1}^N \delta_{y^i},\frac{1}{N}\sum\limits_{i=1}^N \delta_{y^i-Z^i/n}\Big)^2 \leq \frac{C}{n^2},$$ where $C$ depends on the support of $\rho$.

We generate the processes $\big((Y_t^l)_{t \in [0,T]}\big)_{l = 1,\dots,N}$ - the independent twin processes of $(Y_t)_{t \in [0,T]}$. We underline that processes $\big(Y^l_t,b^l_t,\sigma^l_t \big)_{t \in [0,T]},~ l = 1, \dots N$ are i.i.d. $\bP$-a.s. and the random variables $Y_0^l$ are i.i.d $\bP$-a.s. as well.

The technique is as follows: we mollify the empirical projection $u_t^N$ and obtain $u_t^{N,n}$, this way we can take second-order derivatives and afterwards apply the ``propagation of chaos'' argument to approximate $u_t$ by $u_t^N$, namely for any $ t \in [0,T]$ one have $\bP$-a.s.
\begin{align*}
&\sup_{t \in [0,T]}\bE^{1,\dots,N} \Big[|u_t^{N,n}(Y_t^1,\dots,Y_t^N) - u_t(\mu)|\Big] 
\\
&\leq \sup_{t \in [0,T]}\bE^{1,\dots,N}\Big[|u_t^{N,n}(Y_t^1,\dots,Y_t^N) - u_t^{N}(Y_t^1,\dots, Y_t^N)|\Big] 
 + \sup_{t \in [0,T]}\bE^{1,\dots,N}\Big[|u_t(\bar{\mu}^N_t) - u_t(\mu_t)|\Big]\\
&
\leq
\varepsilon_n + \varepsilon_N,
\end{align*}
where $(\varepsilon_k)_{k \geq 1}$ is a sequence of random variables $\bP$-a.s. converging to $0$, as $k \to \infty$ uniformly in time, this is seen via a propagation of chaos argument, continuity of $u$, dominated convergence theorem and the fact that convergence in Wasserstein metric only depends on the moments of the distribution.

By the $\bP$-a.s.~boundedness of $u$ one can get for any $p \geq 1$ 
\begin{align}
\label{eq:EmpirProj-FuncApprox}
\sup_{t \in [0,T]}\bE^{1,\dots,N}\Big[|u_t^{N,n}(Y_t^1,\dots,Y_t^N) - u_t(\mu) |^p\Big]^{\frac{1}{p}} \leq \varepsilon^{(p)}_n+\varepsilon^{(p)}_N,
\end{align}
where $(\varepsilon^{(p)}_k)_{k \in \bN}$ is a sequence converging $\bP$-a.s. to $0$. 

Now we use the Proposition \ref{prop:DerivativeRelations-Space-2-Lions} to get for any $t \in [0,T],~\bP$-a.s.
\begin{align*}
\partial_{y^i}u^{N,n}_t(y^1,\dots,y^N) &= n^{Nd} \int_{(\bR^d)^N} \partial_{y^i}u^N_t(y^1-z^1,\dots,y^N-z^N) \prod\limits_{l=1}^{N} \rho(nz^l)\prod\limits_{l=1}^{N}\dd z^l \\
&
=
\frac{n^{Nd}}{N} \int_{(\bR^d)^N} \partial_\mu u_t \Big(\frac{1}{N}\sum_{l = 0}^N \delta_{y^l-z^l},y^i-z^i\Big) \prod\limits_{l=1}^{N} \rho(nz^l)\prod\limits_{l=1}^{N}\dd z^l.
\end{align*}
Applying the same argument as above we get
$\bP$-a.s., $p \geq 1$ 
\begin{align}
\label{eq:EmpirProj-FirstDerivApprox}
\sup_{t \in [0,T]}\bE^{1,\dots,N}\Big[|N \partial_{y^i} u^{N,n}_t(Y_t^1,\dots,Y_t^N) - \partial_\mu u_t(\mu_t,Y^i_t)|^p\Big]^\frac{1}{p} \leq \varepsilon^{(p)}_n + \varepsilon^{(p)}_N.
\end{align}
Now we differentiate once again with respect to $y^i$ 
\begin{align*}
\partial_{y^i}\partial_{y^i}u^{N,n}_t(y^1,\dots,y^N) &= \frac{n^{Nd+1}}{N} \int_{(\bR^d)^N} \partial_\mu u^N_t \Big(\frac{1}{N}\sum_{l = 1}^N \delta_{y^l-z^l},y^i-z^i\Big) \otimes \partial_{z^i} \rho(nz^i)\prod\limits_{l \ne i}^{N} \rho(nz^l)\prod\limits_{l=1}^{N}\dd z^l,
\end{align*}
with standard tensor product operating on elements of $\bR^d$.

To the previous identity we add and subtract a perturbation term focusing on the contribution by $\delta_{y^i}$
\begin{align*}
N &\partial^2_{y^iy^i} u^{N,n}_t(y^1,\dots,y^N) \\
&
=
n^{Nd+1} \int_{(\bR^d)^N} \partial_\mu u_t \Big(\frac{1}{N}\sum_{l \ne i}^N \delta_{y^l-z^l} + \frac1N \delta_{y^i},y^i-z^i\Big) \otimes \partial_{z^i} \rho(nz^i)\prod\limits_{l \ne i}^{N} \rho(nz^l)\prod\limits_{l=1}^{N}\dd z^l \\
&
+
n^{Nd+1} \int_{(\bR^d)^N} \bigg[\partial_\mu u_t \Big(\frac{1}{N}\sum_{l = 1}^N \delta_{y^l-z^l},y^i-z^i\Big) - \partial_\mu u_t \Big(\frac{1}{N}\sum_{l \ne i}^N \delta_{y^l-z^l} +  \frac1N \delta_{y^i},y^i-z^i\Big)\bigg] \\
&
\otimes 
\partial_{z^i} \rho(nz^i)\prod\limits_{l \ne i}^{N} \rho(nz^l)\prod\limits_{l=1}^{N}\dd z^l \\
&
=
T^{1,N}_{n,i}(y^1,\dots,y^N) + T^{2,N}_{n,i}(y^1,\dots,y^N).
\end{align*}
We integrate by parts $T^{1,N}_{n,i}$ with respect to the space variable $y$ (that appears from the derivative in measure and notice the two minus signs), use the compact support of $\rho$ for the boundary term, and to the resulting integral term we add and subtract a $\partial_v\partial_\mu u^N_t$ over the whole empirical measure, this yields
\begin{align*}
T^{1,N}_{n,i}(y^1,\dots,y^N) 
&
= n^{Nd} \int_{(\bR^d)^N} \partial_v\partial_\mu u_t \Big(\frac{1}{N}\sum_{l = 1}^N \delta_{y^l-z^l},y^i-z^i\Big) \prod\limits_{l=1}^{N} \rho(nz^l)\prod\limits_{l=1}^{N}\dd z^l\\
&
\qquad
+
n^{Nd} \int_{(\bR^d)^N} \bigg[ \partial_v\partial_\mu u_t \Big(\frac{1}{N}\sum_{l \ne i}^N \delta_{y^l-z^l} + \frac{1}{N}\delta_{y^i},y^i-z^i\Big)\\
&
\qquad
-
\partial_v\partial_\mu u_t \Big(\frac{1}{N}\sum_{l = 1}^N \delta_{y^l-z^l},y^i-z^i\Big)
\bigg] \prod\limits_{l = 1}^{N} \rho(nz^l)\prod\limits_{l=1}^{N}\dd z^l\\
&
=
T^{11,N}_{n,i}(y^1,\dots,y^N) + T^{12,N}_{n,i}(y^1,\dots,y^N).
\end{align*}
For $T^{11,N}_{n,i}$ we have, as previously due to uniform continuity of $\partial_v\partial_\mu u_t$,  $\bP$-a.s.,
$p \geq 1$ 
\begin{align}
\label{eq:EmpirProj-T11Approx}
\sup_{t \in [0,T]}\bE^{1,\dots,N}\Big[|T^{11,N}_{n,i}(Y_t^1,\dots,Y_t^N)  - \partial_v \partial_\mu u_t(\mu_t,Y^i_t)|^p\Big]^\frac{1}{p} \leq \varepsilon^{(p)}_n + \varepsilon^{(p)}_N.
\end{align}
Uniform continuity of $\partial_v \partial_\mu u$ (in space-measure variables) together with the properties of the Wasserstein metric over finitely supported measures gives
$$W_2 \Big(\frac{1}{N}\sum\limits_{l\ne i}^N \delta_{y^l-z^l} + \frac{1}{N}\delta_{y^i},\frac{1}{N}\sum\limits_{l=1}^N \delta_{y^l-z^l}\Big)^2 \leq \frac{1}{N}C,$$
which in turn implies $\bP$-a.s., $p \geq 1$
\begin{align}
\label{eq:EmpirProj-T12Approx}  
\sup_{t \in [0,T]}\bE^{1,\dots,N}\big[|T^{12,N}_{n,i}(Y_t^1,\dots,Y_t^N)|^p\big]^{\frac{1}{p}} \leq \varepsilon^{(p)}_N.
\end{align}
The procedure to deal with $T^{12,N}_{n,i}$ also applies to $T^{2,N}_{n,i}$ and yields $\bP$-a.s.~for any $t \in [0,T]$
\begin{align}
\label{eq:EmpirProj-T2Approx}    
\sup_{t \in [0,T]}\bE^{1,\dots,N}\big[|T^{2,N}_{n,i}(Y_t^1,\dots,Y_t^N)|^p\big]^{\frac{1}{p}} \leq n\varepsilon^{(p)}_N,
\end{align}
with an additional multiplicative factor $n$ appearing after differentiating the regularisation kernel.

We say that $\phi_t(\cdot), \psi_t(\cdot) = 0$ for all other $t$, where $\phi,\psi$ are not defined.
Now the same technique is valid to $\phi^{N,n}, \psi^{N,n}$ to get $\bP$-a.s for almost all $t$
\begin{align*}
    \bE^{1,\dots,N}\Big[\phi_t^{N,n}(Y_t^1, \dots, Y_t^N)\Big] &\to \phi_t(\mu_t), \quad \text{as }N,n \to \infty,\\
    \bE^{1,\dots,N}\Big[\psi_t^{N,n}(Y_t^1, \dots, Y_t^N)\Big] &\to \psi_t(\mu_t), \quad \text{as }N,n \to \infty.
\end{align*}

Hence, 
$\bP$-a.s., $p \geq 1$ 
\begin{align}
\label{eq:EmpirProj-PhiApprox}
&\sup_{0\leq t \leq T}\bE^{1,\dots, N}\Big[ |\int_0^t\phi_s^{N,n}(Y^1_s,\dots,Y^N_s)\dd s - \int_0^t\phi_s(\mu_s)\dd s|^p\Big]^{\frac{1}{p}} \leq \varepsilon^{(p)}_n+\varepsilon^{(p)}_N,\\
\label{eq:EmpirProj-PsiApprox}
&\sup_{0\leq t \leq T}\bE^{1,\dots, N}\Big[ |\int_0^t\psi_s^{N,n}(Y^1_s,\dots,Y^N_s)\cdot \dd W_s - \int_0^t\psi_s(\mu_s)\cdot \dd W_s|^p\Big]^{\frac{1}{p}}  \leq \varepsilon^{(p)}_n+\varepsilon^{(p)}_N.
\end{align}
Without loss of generality we pick the $(\varepsilon_k)_{k \in \bN}$ the same as for $u$.
One can notice that $\psi^{N,n}$, $\phi^{N,n}$ satisfy condition \eqref{eq:RedudeIntegrabilityClassicItoWentzell} of Theorem \ref{theo:ExtendItoWentzell}, due to mollification and the identification from Proposition \ref{prop:DerivativeRelations-Space-2-Lions}.

\emph{Step 3: Applying the classical It\^o-Wentzell to the approximation.} Under our assumptions and the mollification argument in combination with Proposition \ref{prop:DerivativeRelations-Space-2-Lions}, we have sufficient regularity that we can apply the standard It\^o-Wentzell formula (see Theorem \ref{theo-Classic-Ito-Wentzell} and Theorem \ref{theo:ExtendItoWentzell}) to $u^{N,n}$ and obtain
\begin{align}
\nonumber
0 =~ &u^{N,n}_t(Y^1_t,\dots,Y^N_t) - u^{N,n}_t(Y^1_0,\dots,Y^N_0)\\
&
\nonumber
-
\int_0^t \phi_s^{N,n} (Y_s^1,\dots,Y^N_s)\dd s - \int_0^t \psi^{N,n}_s (Y_s^1,\dots,Y^N_s) \cdot \dd W_s\\
&
\label{eq:ItoWentzellParticleSystem-FullFlow}
-
\frac{1}{N} \sum_{l=1}^N \int_0^t \partial_{y^l} u^{N,n}_s(Y_s^1,\dots,Y^N_s) \cdot b^l_s \dd s\\
&
\nonumber
-
\frac{1}{N} \sum_{l=1}^N \int_0^t \partial_{y^l} u^{N,n}_s(Y_s^1,\dots,Y^N_s) \cdot \sigma^l_s \dd W^l_s\\
&
\nonumber
-
\frac{1}{2N} \sum_{l=1}^N \int_0^t \trace \big\{\partial^2_{y^ly^l} u^{N,n}_s(Y_s^1,\dots,Y^N_s) ~ \sigma^l_s(\sigma_s^l)^{\intercal}\big\} \dd s.
\end{align}
Note two important simplifications. Firstly, one would expect the second-derivative term to contain a Hessian, but for independent processes $Y^{l_1}_t,Y^{l_2}_t, l_1 \neq l_2$, we have  $\dd\langle  Y^{l_1},  Y^{l_2} \rangle_t = \1_{\{l_1=l_2\}}\sigma_t^{l_1}(\sigma^{l_2}_t)^\intercal \dd t$ and hence only diagonal terms appear. Secondly, no cross-variation term $d\langle \partial_\mu u^{N,n}, Y^l \rangle_t$ appears, this is due to the independence of the field's noise $W_t$ and noise of the particles $\{W^l_t\}_{l=1,\dots,N}$ within empirical approximation (this will not be the case in the next section).

Now we can proceed with the expected result. Define $\Delta^{N,n}$ as the difference between the RHS of \eqref{eq:ItoWentzellParticleSystem-FullFlow} and the RHS of the below equation, we then have for any $t\in[0,T]$ $\bP$-a.s. (the tautology)  
\begin{align*}
\Delta_t^{N,n} &= u_t(\mu_t) - u_0(\mu_0)
- \int_0^t\phi_s(\mu_s) \dd s - \int_0^t\psi_s(\mu_s) \cdot \dd W_s
-
\frac{1}{N}\sum_{l = 1}^N\int_0^t \partial_\mu u_s(\mu_s,Y^l_s) \cdot b^l_s \dd s \\
&
-\frac{1}{N} \sum_{l=1}^N \int_0^t \partial_{\mu} u_s(\mu_s,Y^l_s)  \cdot\sigma^l_s \dd W^l_s
- 
\frac{1}{2N} \sum_{l = 1}^N \int_{0}^{t} \trace \big\{\partial_v\partial_\mu u_s(\mu_s, Y^l_s) ~\sigma^l_s(\sigma^l_s)^{\intercal} \big\} \dd s.
\end{align*}
It is clear that $[0,T]\ni t\mapsto \Delta_t^{N,n}$ is continuous. Moreover, collecting the inequalities  \eqref{eq:EmpirProj-FuncApprox}-\eqref{eq:EmpirProj-PsiApprox}  we have $\sup_{0\leq t \leq T}|\bE^{1,\dots, N}\big[\Delta_t^{N,n} \big]|\leq \varepsilon_n + (1+n)\varepsilon_N$,  $\bP$-a.s.

We let $N \to \infty$ to get by Fatou's lemma, the law of large numbers and the joint-continuity of all derivatives with localisation argument for stochastic integral term, $\bP$-a.s. that $\sup_{0\leq t \leq T} |\Delta^n_t| \leq 3\varepsilon_n$, where $\bP$-a.s.
\begin{align}
\nonumber
\Delta^n_t &= u_t(\mu_t) - u_0(\mu_0) -\int_0^t\phi_s(\mu_s) \dd s - \int_0^t\psi_s(\mu_s) \cdot \dd W_s
\\
\label{eq:aux-for-measurability-discussion}
&
-
\int_0^t \tilde{\bE} \Big[ \partial_\mu u_s(\mu_s,\tilde{Y}_s) \cdot \tilde{b}_s \Big] \dd s 
-
\frac{1}{2}\: \int_{0}^{t}\tilde{\bE}  \Big[  \trace \big\{\partial_v\partial_\mu u_s(\mu_s,\tilde{Y}_s)  ~
\tilde{\sigma}_s(\tilde{\sigma}_s)^{\intercal}\big\} \Big] \dd s,
\end{align}
where $\Delta^n_t := \lim_{n \to \infty} \Delta^{N,n}_t$, and we applied Fubini's theorem to interchange the Lebesgue integral with the expectation.  
Note that to handle the stochastic integral we apply the localisation technique and use dominated convergence theorem once more.
Letting $n \to \infty$ in the equation above, we conclude that $\Delta^n_t \to \Delta \equiv 0,~\bP$-a.s., which finishes this part of the proof. The measurability of the involved coefficients follows the guidelines set in Remark \ref{rem:measurability-main-remark}.
\end{proof}

\subsection{The joint chain rule}
 
Now we are ready to provide a joint chain rule formula expanding the nature of the random field to support a space variable dependence, i.e. the case $t\mapsto u_t(X_t,\mu_t)$ for $\mu$ the law of \eqref{eq:GenericYprocess-FullFlow22} and $X$ solution to \eqref{eq:ProcessX-fullFlowMeasures}. Let us start by inheriting the structure and properties of the setup of Theorem \ref{theo:IW-FullFlow-measureOnly}. 

\begin{definition}
\label{def:PartialC2-RandomField-FullFlow-ExtendedChainRule}
We say the random field $u:\Omega\times [0,T] \times \bR^d \times \cP_2(\bR^d)\to \bR$ given in \eqref{eq:RandomField-1BM} is \emph{RF-Joint-Partially}-$\cC^{2}$ if 
\begin{enumerate}[i)]
    \item For any $(x,\mu) \in \bR^d \times\cP_2(\bR^d)$,  $(u_t(x,\mu))_{t\in[0,T]}$ is a continuous adapted process taking values in $\bR$ and $(\phi_t(x,\mu))_{t \in [0,T]},(\psi_t(x,\mu))_{t \in [0,T]}$ are $\bF$-progressively measurable processes with values in $\bR$ and $\bR^{d}$ respectively;
    
    \item For almost any $t \in[0,T]$, the maps $(x,\mu)\mapsto \phi_t(x,\mu),~(x,\mu)\mapsto \psi_t(x,\mu)$ are $\bP$-a.s.~jointly-continuous in the product topology of $\bR^d \times \cP_2(\bR^d)$ at every pair $(x,\mu) \in \bR^d \times \cP_2(\bR^d)$;
    
    \item For any $(t,\mu) \in [0,T] \times\cP_2(\bR^d)$, the map $x \mapsto u_t(x,\mu)$ is $\cC^2(\bR^d),~\bP$-a.s. at every $x \in \bR^d$, with $\partial_x u, \partial^2_{xx} u$ being $\bP$-a.s.~joint-continuous at every triple $(t,x,\mu) \in [0,T] \times \bR^d \times \cP_2(\bR^d), ~\bP$-a.s.;
    
    \item For almost any $t \in [0,T],$ for any $\mu \in \cP_2(\bR^d)$, the map $x\mapsto \psi_t(x,\mu)$ is $\cC^1(\bR^d),~\bP$-a.s. at every $x \in \bR^d$, with $\partial_x \psi$ being $ \bP$-a.s.~joint-continuous at every pair $(x,\mu) \in \bR^d \times \cP_2(\bR^d),~\bP$-a.s.;
    
    \item For any $(t,x) \in [0,T] \times \bR^d$, the map $\mu \mapsto u_t(x,\mu)$ is $\bP$-a.s.~continuous in the Wasserstein metric and $L$-differentiable $\bP$-a.s. at every $\mu \in \cP_2(\bR^d)$. Moreover, $\partial_\mu u_t(x,\mu,v)$ has a $\mu$-version such that $\partial_\mu u_t(x,\mu,v)$ is $\bP$-a.s.~joint-continuous at every quadruple $(t,x,\mu,v)$, with $(t,x,\mu) \in [0,T] \times\bR^d \times \cP_2(\bR^d),~v \in \Supp(\mu), ~\bP$-a.s.;

    \item For any $(t,x,\mu) \in [0,T] \times \bR^d \times \cP_2(\bR^d)$, the map $v \mapsto\partial_\mu u_t(x,\mu,v)$ is $\bR^d$-differentiable $\bP$-a.s., at every $v \in \Supp(\mu)$. Moreover, $\partial_v\partial_\mu u_t(x,\mu,v)$ has a $\mu$-version such that $\partial_v\partial_\mu u_t(x,\mu,v)$ is $\bP$-a.s.~joint-continuous at every quadruple $(t,x,\mu,v)$, with $(t,x,\mu) \in [0,T] \times\bR^d \times \cP_2(\bR^d),~v \in \Supp(\mu), ~\bP$-a.s..
\end{enumerate}
\end{definition}

\begin{theorem}
\label{theo:IW-FUllmeasureFlow-Proc+Measure-1BM}
Let $u : \Omega \times [0,T] \times\bR^d\times \cP_2(\bR^d)\to \bR$ defined by \eqref{eq:RandomField-1BM} to be \emph{RF-Joint-Partially-$\cC^2$}. Assume that for any compact $K\subset \bR^d\times \cP_2(\bR^d)$ and $t \in [0,T]$ we have

\begin{align}
	\label{cond:Full-Flow-IntRF-FullChainRulePhi}
	\int_0^t \sup_{(x,\mu) \in K}\Big\{|&\phi_s(x,\mu)| 
	+ |\psi_s(x,\mu)|^2 
	+ |\partial_x \psi_s(x,\mu)|^2\Big\}\dd s  < \infty, \quad \bP\text{-a.s.},
\end{align}
and 
\begin{align}
	\label{cond:Full-Flow-IntRF-FullChainRulePsi}
	 \sup_{(t,x,\mu) \in [0,T] \times K}\Big\{\int_{\bR^d}\Big[|\partial_\mu u_t(x,\mu, v)|^2 + |\partial_v\partial_\mu u_t(x,\mu, v)|^2\Big]\mu(\dd v)\Big\}  < \infty, \quad \bP\text{-a.s.}.
\end{align}
Let $(\mu_t)_{t \in [0,T]}$ be the law of the solution to \eqref{eq:GenericYprocess-FullFlow22} satisfying Assumption \ref{Assump:SDE-Y-mu-1BM}. Let $(X_t)_{t \in [0,T]}$ be the solution process to  \eqref{eq:ProcessX-fullFlowMeasures} under Assumption \ref{Assump:SDE-X-1BM}. 

Then the process $(u_t(X_t,\mu_t))_{t\in[0,T]}$ is an It\^o process $\bP$-a.s.~satisfying the dynamics 
\begin{align}
\label{eq:Full-Flow-Full-Chain-Rule}
\nonumber
u_T(X_T&,\mu_T) - u_0(X_0,\mu_0)\\
&
=
\nonumber
\int_0^T\phi_s(X_s,\mu_s)\dd s + \int_0^T\psi_s(X_s,\mu_s) \cdot \dd W_s \\
&
+ \int_0^T \partial_x u_s(X_s,\mu_s) \cdot \gamma_s\dd W_s 
+ \int_0^T \trace\big \{\partial_x\psi_s(X_s,\mu_s)  (\gamma_s)^\intercal \big\} \dd s \\
\nonumber
&
+ \int_0^T \partial_xu_s(X_s,\mu_s) \cdot\beta_s\dd s 
+ \int_0^T \frac{1}{2}~ \trace\big\{\partial_{xx}^2 u_s(X_s,\mu_s)  ~\gamma_s (\gamma_s)^\intercal \big\}\dd s \\
&
\nonumber
+ \int_0^T \tilde{\bE} \Big[\partial_\mu u_s(X_s,\mu_s,\tilde{Y}_s) \cdot \tilde{b}_s\Big]\dd s 
+ \int_{0}^{T} \frac{1}{2} ~\tilde{\bE}\Big[  \trace\big \{\partial_v\partial_\mu u_s(X_s,\mu_s, \tilde{Y}_s)  ~\tilde{\sigma}_s(\tilde{\sigma_s})^{\intercal}\big\}\Big] \dd s,
\end{align}
with $u_0(X_0,\mu_0) = f(X_0,\mu_0)$.
\end{theorem}
Observe that the terms of the first and the last line on the RHS of the formula are the ones from our Theorem \ref{theo:IW-FullFlow-measureOnly}, whereas the middle two arise from the standard It\^o-Wentzell formulae.

\begin{proof}
In view of the proof of Theorem \ref{theo:IW-FullFlow-measureOnly} we assume a compactification/mollification argument in the measure component has been applied. In this way we avoid a repetition of arguments.

We start by fixing a time $T$ and let $\Pi_K=\{0=t_0<t_1<\cdots<t_K=T\}$ be a partition of $[0,T]$ with modulus $|\Pi_K|=\min_{0\leq j\leq K-1}|t_{j+1}-t_j|>0$. Then 
\begin{align*}
& u_T(X_T,\mu_T) - u_0(X_0) 
\\
&= \sum_{i=0}^{K-1} \Big[ u_\tip(X_\tip,\mu_\tip)-u_\ti(X_\ti,\mu_\ti)\Big] \\
&= \sum_{i=0}^{K-1} \Big[ u_\tip(X_\tip,\mu_\tip)-u_\ti(X_\ti,\mu_\tip)\Big] + \sum_{i=0}^{K-1} \Big[ u_\ti(X_\ti,\mu_\tip)-u_\ti(X_\ti,\mu_\ti)\Big] 
= I_1^{(K)}+I_2^{(K)}.
\end{align*}
Now we see that $I_2^{(K)}$ is amenable to Remark \ref{rem:IgnoreItoWentzelAndJustApplyLions} which together with the joint time-space continuity of the measure derivatives, a localisation procedure for $X$, applying twice the dominated convergence theorem in combination with Assumption \ref{Assump:SDE-Y-mu-1BM} yields
\begin{align*}
I_2^{(K)} &= \sum_{i=0}^{K-1}\bigg[\int_\ti^\tip\tilde{\bE}[ \partial_\mu u_\ti(X_\ti,\mu_s,\tilde{Y}_s) \cdot \tilde{b}_s ]\dd s 
+ \frac{1}{2}\:\int_\ti^\tip \tilde{\bE} \Big[\trace\big\{ \partial_v\partial_\mu u_\ti(X_\ti,\mu_s, \tilde{Y}_s)\tilde{\sigma}_s(\tilde{\sigma}_s)^{\intercal} \big\}\Big]\dd s\bigg],\\
&
\to
\int_{0}^{T}\tilde{\bE}[ \partial_\mu u_s(X_s,\mu_s,\tilde{Y}_s) \cdot \tilde{b}_s ]\dd s +
\frac{1}{2} \:\int_{0}^{T} \tilde{\bE} \Big[ \trace\big\{\partial_v\partial_\mu u_s(X_s,\mu_s, \tilde{Y}_s) ~ \tilde{\sigma}_s(\tilde{\sigma}_s)^{\intercal}\big\} \Big]\dd s,
\end{align*}

where we have taken the limit $|\Delta_K|\to 0$.

The measure increment is forward in time for $I_1^{(K)}$, however its flow is deterministic allowing to directly pass to the limit, after applying the Theorem \ref{theo:ExtendItoWentzell}, whose assumptions are satisfied, having
\begin{align*}
I_1^{(K)} &=  \sum_{i=0}^{K-1} \bigg[\int_\ti^\tip \phi_s(X_s,\mu_\tip)\dd s +\int_\ti^\tip \psi_s(X_s,\mu_\tip) \cdot \dd W_s \\
&+ \int_\ti^\tip \partial_x u_s(X_s,\mu_\tip) \cdot \beta_s \dd s + \int_\ti^\tip \partial_x u_s(X_s,\mu_\tip)  \cdot \gamma_s \dd W_s\big]  \\
&+ \frac12\int_\ti^\tip\trace\big\{\partial^2_{xx}u_s(X_s,\mu_\tip) ~\gamma_s(\gamma_s)^\intercal\big\} \dd s + \int_\ti^\tip \trace\big \{\partial_x \psi_s(X_s,\mu_\tip) (\gamma_s)^\intercal \big\} \dd s\bigg].
\end{align*}
Now one can pass to the limit in $I_1^{(K)}$ as $|\Delta_K|\to 0$, by applying joint-continuity of $u$ and its derivatives, alongside Lebesgue dominated convergence theorem, localisation procedure to deal with $X$, and standard quadratic variation argument to handle stochastic integral, so

\begin{align}
\nonumber
I_{1}^{(K)} &\to \int_0^T\phi_s(X_s,\mu_s)\dd s +\int_0^T \psi_s(X_s,\mu_s)\cdot \dd W_s \\
\label{eq:Convergence-Of-ForwardMeasureFlow}
&+ \int_0^T \partial_x u_s(X_s,\mu_s) \cdot \beta_s \dd s + \int_0^T \partial_x u_s(X_s,\mu_s) \cdot \gamma_s \dd W_s\big]  \\
\nonumber
&+ \frac12\int_0^T\trace\{\partial^2_{xx}u_s(X_s,\mu_s) ~\gamma_s(\gamma_s)^\intercal\big\} \dd s + \int_0^T \trace \big\{\partial_x \psi_s(X_s,\mu_s) (\gamma_s)^\intercal\big\} \dd s.
\end{align}
Joining all the limits we see that \eqref{eq:Full-Flow-Full-Chain-Rule} immediately follows.
Measurability is dealt by Remark \ref{rem:measurability-main-remark}.

\end{proof}


%
%
%
%

\section{It\^o-Wentzell-Lions chain rule with a conditional flow of measures}
\label{sec:conditionalflow}

The setting discussed in this section is inspired by the developments in the theory of mean-field games with common noise, \cite{cardaliaguet2019master} and \cite{CarmonaDelarue2017book2}. Since the framework evolves from that in the previous sections we set up our probability spaces and notation anew. 

We consider $(\Omega^0,\cF^0, \bF^0 = (\cF^0_t)_{t \in [0,T]},\bP^0)$ and $(\Omega^1,\cF^1, \bF^1 = (\cF^1_t)_{t \in [0,T]},\bP^1)$ atomless Polish probability spaces to be the respective completions of $(\Omega^0,\bF^0,\bP^0)$ and $(\Omega^1,\bF^1,\bP^1)$ carrying a respective $d$-dimensional Brownian motions $W^0 = (W_t^0)_{t \in [0,T]}$ and $W^1 = (W_t^1)_{t \in [0,T]}$ generating the probability space's filtration, augmented by all $\bP^0$- and $\bP^1$-null sets respectively. We augment $(\Omega^0,\cF^0, \bF^0 = (\cF^0_t)_{t \in [0,T]},\bP^0)$ with a sufficiently rich sub $\sigma$-algebra $\cF^0_0$ independent of $W^0$ and $W^1$. We denote by $(\Omega,\bF, \bP)$ the completion of the product space $(\Omega^0 \times \Omega^1,\bF^0 \otimes \bF^1, \bP^0 \otimes\bP^1)$ equipped with the filtration $\bF$ obtained by augmenting the product filtration $\bF^0 \otimes \bF^1$ in a right-continuous way and by completing it. In the vein of Section \ref{sec:LionsDerivative} let $\bE^0$ and $\bE^1$ taking the expectation on the first and second space respectively.

Let $u : \Omega \times [0,T] \times\bR^d\times \cP_2(\bR^d)\to \bR$ be a random field, satisfying the dynamics 
\begin{align}
\label{eq:RandomField-2BM}
\dd u_t(x,\mu) = \phi_t(x,\mu)\dd t+\psi^0_t(x,\mu)\cdot \dd W^0_t + \psi^1_t(x,\mu)\cdot \dd W^1_t,\quad u_0(x,\mu) = f(x,\mu),
\end{align}
where $f(x,\mu): \bR^d \times \cP_2(\bR^d) \to \bR$ is a deterministic function, $W^0 = (W^0_t)_{t \in [0,T]}$ and $W^1=(W^1_t)_{t \in [0,T]}$ are independent $d$-dimensional $\bF^0$ and $\bF^1$-Brownian motions respectively; $(\phi, \psi^0,\psi^1): \Omega\times [0,T] \times \bR^d \times \cP_2(\bR^d) \to \bR \times \bR^{d} \times \bR^{d}$ are $\bF$-progressively measurable processes. 
 
Take measurable $(b,\sigma^0, \sigma^1): \Omega \times [0,T] \to \bR^d \times \bR^{d \times d}\times \bR^{d \times d}$ and define the following process
\begin{align}
\label{eq:GenericYprocess-PartialFlow22}
\dd Y_t = b_t\dd t + \sigma^0_t \dd W^0_t+ \sigma^1_t \dd W^1_t, \text{ and initial condition } Y_0\in L^2(\Omega,\cF_0,\bP),
\end{align}
and $\mu_t := \Law(Y_t(\omega_0,\cdot))$ for $\bP^0$-almost any $\omega_0$.
Here $\Law(Y_t(\omega_0,\cdot))$ can be understood as RV from $(\Omega^0,\cF^0,\bP^0)$ into $\cP(\bR^d)$ (for further details see discussion in \cite{CarmonaDelarue2017book2}*{Section 4.3}).

Moreover, the involved coefficients satisfy the next conditions
\begin{assumption}
\label{Assump:SDE-Y-mu-2BM}
$(Y_t)_{t \in [0,T]}$ satisfies Assumption \ref{Assump:SDE-Y-mu-1BM} with $\sigma_t := \left( \begin{smallmatrix} \sigma_t^0 & 0 \\ 0 & \sigma_t^1 \end{smallmatrix} \right)$ and $W_t := (W^0_t,W^1_t)^\intercal$.
\end{assumption}

Take $(X_t)_{t \in [0,T]}$ satisfying dynamics
\begin{align}
\label{eq:ProcessX-partialFlowMeasures}
\dd X_t = \beta_t\dd t +\gamma^0_t\dd W^0_t+\gamma^1_t\dd W^1_t, \text{ and initial condition } X_0 \in L^2(\Omega,\cF_0,\bP),
\end{align}
with coefficients satisfying
\begin{assumption}
\label{Assump:SDE-X-2BM}
$(X_t)_{t \in [0,T]}$ satisfies Assumption \ref{Assump:SDE-X-1BM} with $\gamma_t := \left( \begin{smallmatrix} \gamma_t^0 & 0 \\ 0 & \gamma_t^1 \end{smallmatrix} \right)$ and $W_t := (W^0_t,W^1_t)^\intercal.$
\end{assumption}

We name $(W_t^0)_{t \in [0,T]}$ as a common noise affecting the whole setting, whilst $(W^1)_{t \in [0,T]}$ is the idiosyncratic chaos for the random field and all processes within.
For the purposes of the present section we fix the common noise and derive the dynamics of the random field by conditioning on $W^0$. Once again, all measurability issues are discussed at Remark  \ref{rem:measurability-main-remark}.


\subsection{It\^o-Lions chain rule along a conditional flow of measures (classic)}

We recall the It\^o-Lions formula for the flow of marginals \cite{CarmonaDelarue2017book2}*{Theorem 4.17}.

First, we provide the regularity assumption as given in \cite{CarmonaDelarue2017book2}*{Subsection 4.3.4}.

\begin{definition}
A function $u:[0,T] \times \bR^d \times \cP_2(\bR^d)\to \bR$ is $\cC^{1,2,(2)}$ if
\begin{enumerate}[i)]
    \item For any $\mu\in\cP_2(\bR^d)$ the map $[0,T]\times \bR^d \ni (t,x)\mapsto u_t(x,\mu)$ is $\cC^{1,2}$, and the maps $\partial_t u$, $\partial_x u$ and $\partial_{xx}^2 u$ are joint-continuous at every triple $ (t,x,\mu) \in [0,T] \times \bR^d \times \cP_2(\bR^d)$;

	\item For any $(t,x) \in [0,T]\times \bR^d$, the map $\mu \mapsto u_t(x,\mu)$ is continuously L-differentiable at every point $\mu \in \cP_2(\bR^d)$. Moreover, $\partial_\mu u: [0,T] \times \bR^d \times \cP_2(\bR^d) \times \bR^d \to \bR^d$ has a $\mu$-version such that $\partial_\mu u_t(x,\mu,v)$ is joint-continuous and locally bounded at every quadruple $(t,x,\mu,v)$, with $(t,x,\mu) \in [0,T] \times \bR^d \times \cP_2(\bR^d),~ v \in \Supp(\mu)$;
	
	\item For any $(t,x,\mu)\in [0,T] \times \bR^d \times \cP_2(\bR^d)$, the map $v \mapsto \partial_\mu u_t(x,\mu,v)$ is continuously $\bR^d$-differentiable at every point $v \in \bR^d$. Moreover, its derivative $\partial_v \partial_\mu u: [0,T] \times \bR^d \times \cP_2(\bR^d) \times \bR^d \to \bR^{d\times d}$ has a $\mu$-version such that $\partial_v\partial_\mu u_t(x,\mu,v)$ is continuous and locally bounded at every quadruple $(t,x,\mu,v)$, with $(t,x,\mu) \in [0,T] \times \bR^d \times \cP_2(\bR^d),~ v \in \Supp(\mu)$;
	
	\item For any $(t,x,\mu,v)\in [0,T] \times \bR^d \times \cP_2(\bR^d) \times \Supp(\mu)$, the map $\nu \mapsto \partial_\mu u_t(x,\mu,v)$ is continuously L-differentiable at every point $\mu \in \cP_2(\bR^d)$. Moreover, its derivative $\partial^2_\mu u: [0,T] \times \bR^d \times \cP_2(\bR^d) \times \bR^d \times \bR^d \to \bR^{d\times d}$ has a $\mu$-version such that $\partial^2_\mu u_t(x,\mu,v)$ is continuous and locally bounded at every quintuple $(t,x,\mu,v,v^\prime)$, with $(t,x,\mu) \in [0,T] \times \bR^d \times \cP_2(\bR^d),~ v,v^\prime \in \Supp(\mu)$;
	
	\item For any $(t,x,\mu,v)\in [0,T] \times \bR^d \times \cP_2(\bR^d)\times \Supp(\nu)$, the map $x \mapsto \partial_\mu u_t(x,\mu,v)$ is continuously $\bR^d$-differentiable at every point $x \in \bR^d$. Moreover, its derivative $\partial_x \partial_\mu u: [0,T] \times \bR^d \times \cP_2(\bR^d) \times \bR^d \to \bR^{d\times d}$ has a $\mu$-version such that $\partial_x \partial_\mu u_t(x,\mu,v)$ is continuous and locally bounded at every quadruple $(t,x,\mu,v)$, with $(t,x,\mu) \in [0,T] \times \bR^d \times \cP_2(\bR^d),~ v \in \Supp(\mu)$;
\end{enumerate}
\end{definition}

\begin{theorem}
\label{theo-Classic-Ito-Lions-PartialFlow}
Let $u:[0,T] \times \bR^d \times \cP_2(\bR^d) \to \bR$ be $\cC^{1,2,(2)}$. Furthermore for any compact $K \subset \bR^d \times \cP_2(\bR^d)$ we have
\begin{align*}
	\sup_{(t,x,\mu) \in [0,T]\times K}\bigg\{\int_{\bR^d}\Big[|\partial_\mu u_t(x,\mu, v)|^2 &+ |\partial_v\partial_\mu u_t(x,\mu, v)|^2
	+ |\partial_x\partial_\mu u_t(x,\mu, v)|^2\Big]\mu(\dd v) 
	\\
	& + \int_{\bR^d \times \bR^d} \Big[  |\partial^2_{\mu} u_t(x,\mu, v,v^\prime)|^2\Big] \mu(\dd v)\mu(\dd v^\prime)\bigg\}  < \infty, \quad \bP \textit{-a.s.}.
\end{align*}

Take $(\mu_t)_{t \in [0,T]}$ associated to \eqref{eq:GenericYprocess-PartialFlow22} under Assumption \ref{Assump:SDE-Y-mu-2BM}. 
Take $(X_t)_{t \in [0,T]}$ to be a $d$-dimensional It\^o process with dynamics \eqref{eq:ProcessX-partialFlowMeasures} satisfying Assumption \ref{Assump:SDE-X-2BM}. 

Then $(u_t(X_t,\mu_t))_{t\in[0,T]}$ is an It\^o process satisfying $\bP$-a.s.
\begin{align*}
u_T(X_T,\mu_T) ~-~ &u_0(X_0,\mu_0) 
\\
&
= 
\int_0^T \partial_t u_s(X_s,\mu_s) \dd s  
+\int_0^T \partial_xu_s(X_s,\mu_s) \cdot \beta_s \dd s 
+ \int_0^T \partial_xu_s(X_s,\mu_s) \cdot \gamma^0_s \dd W^0_s \\
&
+ \int_0^T \partial_xu_s(X_s,\mu_s) \cdot \gamma^1_s \dd W^1_s 
+ \int_0^T \frac{1}{2} ~\trace\big\{\partial_{xx}^2 u_s(X_s,\mu_s)  (\gamma^0_s (\gamma^0_s)^\intercal +\gamma^1_s (\gamma^1_s)^\intercal ) \big\}\dd s \\
&
+ \int_0^T \tilde{\bE}^1 \Big[\partial_\mu u_s(X_s,\mu_s,\tilde{Y}_s)\cdot \tilde{b}_s\Big]\dd s 
+\int_0^T \tilde{\bE}^1 \Big[(\tilde{\sigma}^0_s)^\intercal\partial_\mu u_s(X_s,\mu_s,\tilde{Y}_s) \Big]\cdot \dd W_s^0 \\
&
+ \int_{0}^{T} \frac{1}{2}~ \tilde{\bE}^1\Big[  \trace\big \{\partial_v\partial_\mu u_s(X_s,\mu_s, \tilde{Y}_s)  (\tilde\sigma^0_s(\tilde{\sigma}_s^0)^{\intercal}+ \tilde{\sigma}^1_s(\tilde{\sigma}_s^1)^{\intercal})\big\}\Big] \dd s \\
&
+
\int_{0}^{T}\tilde{\bE}^1\Big[  \trace\big \{\partial_x\partial_\mu u_s(X_s,\mu_s, \tilde{Y}_s) ~ \gamma^1_s(\tilde\sigma_s^1)^{\intercal}\big\}\Big] \dd s
\\
&
+  \int_0^T \frac12 ~\hat{\bE}^1 \Big[\tilde{\bE}^1 \Big[ \trace\big \{\partial^2_{\mu} u_s(X_s,\mu_s,\tilde{Y}_s,\hat{Y}_s) ~\tilde{\sigma}^0_s(\hat \sigma_s^0)^\intercal\big\} \Big]\Big]\dd s
\end{align*}
where $\tilde{\bE}$ denotes the expectation acting on the model twin spaces $(\tilde{\Omega}, \tilde{\bF}, \tilde{\bP})$ and $(\hat{\Omega}, \hat{\bF}, \hat{\bP})$ and let the processes $(\tilde{Y}_t, \tilde{b}_t, \tilde{\sigma}_t)_{t\in[0,T]}$ and $(\hat{Y}_t, \hat{b}_t, \hat{\sigma}_t)_{t\in[0,T]}$ be the twin processes of $(Y_t,b_t,\sigma_t)_{t\in[0,T]}$ respectively living within.
\end{theorem}


\subsection{It\^o-Wentzell-Lions formula for measure functionals}

For the derivation of the expansion in measure component, and as in Theorem \ref{theo:IW-FullFlow-measureOnly}, we fix $x \in \bR^d$ then omit its dependence, i.e. 
\begin{align*}
u_t(\mu):=u_t(x,\mu), \quad
\phi_t(\mu):=\phi_t(x,\mu), 
\quad \psi^0_t(\mu):=\psi^0_t(x,\mu) \quad\textrm{and}\quad
\psi^1_t(\mu):=\psi^1_t(x,\mu). 
\end{align*}
Now we introduce the regularity for random field given by \eqref{eq:RandomField-2BM} which inherits  Definition \ref{def:PartialC2-RandomField-FullFlow} and requires additionally a second-order Fr\'echet differentiability.
\begin{definition}
\label{rate def:PartialC2-RandomField-PartialFlow}
We say the random field $u:\Omega\times [0,T] \times \cP_2(\bR^d)\to \bR$ given in \eqref{eq:RandomField-2BM} (for some $x\in \bR^d$ fixed) is \emph{RF-Generally}-$\cC^{2}$ if 
\begin{enumerate}[i)]
    \item $u$ is \emph{RF-Partially}-$\cC^2$ for $\psi_t :=  (\psi_t^0 ,\psi_t^1)^\intercal$ and $W_t := (W^0_t,W^1_t)^\intercal$;

 	\item For any $(t,v) \in [0,T]\times \Supp(\mu)$, the map $\mu \mapsto \partial_\mu u_t(\mu,v)$ is L-differentiable $\bP$-a.s. at every point $ \mu \in \cP_2(\bR^d)$. Moreover, $\partial^2_{\mu} u_t(\mu,v,v^\prime)$ has a $\mu$-version such that $\partial^2_\mu u_t(\mu,v,v^\prime)$ is $\bP$-a.s.~joint-continuous at every quadruple $(t,\mu,v,v^\prime)$, with $(t,\mu) \in [0;T] \times \cP_2(\bR^d),~v,v^{\prime} \in \Supp(\mu),~\bP$-a.s.;
    
    \item For almost any $t \in [0,T]$, the map $\mu\mapsto \psi^0_t(\mu)$ is L-differentiable $\bP$-a.s. at every point $\mu \in \cP_2(\bR^d)$. Moreover, $\partial_\mu \psi^0_t(\mu,v)$ has a $\mu$-version such that $\partial_\mu \psi^0_t(\mu,v)$ is $\bP$-a.s.~joint-continuous at every pair $(\mu,v),~\mu \in \cP_2(\bR^d),~v \in \Supp(\mu),~\bP$-a.s.
\end{enumerate}
\end{definition}
We highlight the slight abuse of notation in the way point \emph{i)} in the above Definition \ref{rate def:PartialC2-RandomField-PartialFlow} is formulated. This avoids re-stating a full assumption that is nonetheless clear to understand.

\begin{theorem}
\label{theo:IW-PartialFlow-measureOnly}
Let $u$ be \emph{RF-Generally-$\cC^{2}$} It\^o random field \eqref{eq:RandomField-2BM} (where $x\in \bR^d$ is fixed and omitted throughout, also for $\phi$ and $\psi$). Assume for any compact $K\subset \cP_2(\bR^d)$ we have
\begin{align*}
	\int_0^T \sup_{\mu \in K}\Big\{|\phi_s(\mu)| + |\psi^0_s(\mu)|^2 + |\psi^1_s(\mu)|^2  + \int_{\bR^d}  |\partial_\mu \psi^0_s(\mu,v)|^2 \mu(\dd v)\Big\}\dd s < \infty,\quad \bP \textit{-a.s.},
\end{align*}
    and
\begin{align}
    \label{cond:Partial-Flow-IntRF-FullChainRule-U}
	\sup_{(t,\mu) \in [0,T] \times K}\Big\{\int_{\bR^d}\Big[|\partial_\mu u_t(\mu, v)|^2 + |\partial_v\partial_\mu u_t(\mu, v)|^2\Big]\mu(&\dd v) 
	\\
	\nonumber
	+\int_{\bR^d \times \bR^d} |\partial^2_{\mu} u_t(\mu,v,v^\prime|^2 \mu(&\dd v)\mu(\dd v^\prime) \Big\}  < \infty,\quad \bP \textit{-a.s.}.
\end{align}
For almost all $\omega^0 \in \Omega^0$ take $(\mu_t)_{t \in [0,T]} := \big(Law(Y_t(\omega_0,\cdot))\big)_{t \in [0,T]},$ with $Y$ solution to \eqref{eq:GenericYprocess-PartialFlow22} under Assumption \ref{Assump:SDE-Y-mu-2BM}.

Then $(u_t(\mu_t))_{t\in[0,T]}$ is an It\^o process $\bP$-a.s. satisfying the expansion
\begin{align}
\nonumber
u_T(\mu_T) - u_0(\mu_0) =&\int_0^T\phi_s(\mu_s) \dd s + \int_0^T\psi^0_s(\mu_s) \cdot \dd W^0_s+ \int_0^T\psi^1_s(\mu_s)  \cdot \dd W^1_s\\
\nonumber
&
+
\int_0^T \tilde{\bE}^1 \Big[\partial_\mu u_s(\mu_s,\tilde{Y}_s)\cdot \tilde{b}_s\Big] \dd s +\int_0^T \tilde{\bE}^1 \Big[(\tilde{\sigma}^0_s)^\intercal\partial_\mu u_s(\mu_s,\tilde{Y}_s) \Big] \cdot \dd W^0_s \\
\label{eq:PartialFlow-MeasureOnly-Formula}
&
+ 
\int_{0}^{T} \frac{1}{2}~ \tilde{\bE}^1  \Big[ \trace\big\{\partial_v\partial_\mu u_s(\mu_s, \tilde{Y}_s)  ( \tilde{\sigma}^0_s(\tilde{\sigma}_s^0)^{\intercal} + \tilde{\sigma}^1_s(\tilde{\sigma}_s^1)^{\intercal}) \big\}\Big]\dd s \\
\nonumber
&
+ 
\int_{0}^{T} \frac{1}{2}~ \hat{\bE}^1\Big[\tilde{\bE}^1  \Big[\trace\big\{\partial^2_{\mu} u_s(\mu_s, \tilde{Y}_s,\hat{Y_s})~ \tilde{\sigma}^0_s(\hat{\sigma}_s^0)^{\intercal} \big\}\Big]\Big]\dd s \\
\nonumber
&
+
\int_{0}^{T}\tilde{\bE}^1\Big[\trace\big\{\partial_\mu \psi^0_s(\mu_s,\tilde{Y}_s) (\tilde{\sigma}^0_s)^\intercal\big\}\Big] \dd s,
\end{align}
where the formula above $\tilde{\bE}$ and $\hat{\bE}$ denote the expectation acting on the model twin spaces $(\tilde{\Omega}, \tilde{\bF}, \tilde{\bP})$ and $(\hat{\Omega}, \hat{\bF}, \hat{\bP})$ respectively, and let the processes $(\tilde{Y}_t, \tilde{b}_t, \tilde{\sigma}_t)_{t\in[0,T]}$ and $(\hat{Y}_t, \hat{b}_t, \hat{\sigma}_t)_{t\in[0,T]}$ be the independent twin processes of $(Y_t,b_t,\sigma_t)_{t\in[0,T]}$ respectively living within.
\end{theorem}

One can notice two new terms appearing in contrast with the formula in Theorem  \ref{theo:IW-FullFlow-measureOnly}. Whilst the $\partial_\mu^2 u$ term appears as a cross-variation of two model particles $\tilde Y$ and $\hat Y$ experiencing the same noise $W^0$ and is present in Theorem \ref{theo-Classic-Ito-Lions-PartialFlow}, a brand new $\partial_\mu \psi^0$ term now indicates an interaction of the field $u$ with the model particle $\tilde Y$ through the same $W^0$.

In contrast to the proof of Theorem \ref{theo:IW-FullFlow-measureOnly}, the arguments here are far more straightforward. This is due to the fact that we now expect to receive a $\partial_\mu^2 u$ term within the expansion, so we should assume the respective regularity, whilst the same situation in the proof of Theorem \ref{theo:IW-FullFlow-measureOnly} requires another round of mollification.

\begin{proof}[Proof of Theorem \ref{theo:IW-PartialFlow-measureOnly}]

\emph{Step 1. Mollification.}  We carry out mollification in two steps - firstly we construct the mollifying sequence and later show its convergence.
As in the proof of Theorem \ref{theo:IW-FullFlow-measureOnly}, we pick a smooth function $\rho : \bR^d \to \bR^d$ with compact support, letting for any $t \in [0,T]$, $(u\star \rho)_t(\mu) := u_t(\mu \circ \rho^{-1})$ and for any $t \in [0,T]$ having $u$ $\bP$-a.s. bounded and continuous at every pair $(t,\mu) \in \cP_2(\bR^d)$, $\partial_\mu u$ and $\partial_v\partial_\mu u$ $\bP$-a.s. bounded and continuous at every triple
$(t,\mu,v)$ for $v \in \Supp (\mu)$ and $\partial^2_{\mu} u$ $\bP$-a.s. bounded and continuous at every quadruple $(t,\mu,v,v^\prime)$ for $v,v^\prime \in \Supp(\mu)$, what follows from local boundedness of $u$ and its derivatives. 
Now picking the sequence $(\rho_n)_{n \geq 1}$ in a way that $(\rho_n, \partial_x\rho_n, \partial^2_{xx}\rho_n)(x) \to (x, I_d,0)$ as $n\to \infty$, we can conclude that $(u\star \rho_n)_t(\mu),\partial_\mu (u \star \rho_n)_t(\mu,v), \partial_v \partial_\mu (u \star \rho_n)_t(\mu,v)$ and $\partial^2_{\mu} (u \star \rho_n)_t(\mu,v,v^\prime)$ converge $\bP$-a.s. to $u_t(\mu),\partial_\mu u_t (\mu,v)$, $\partial_v \partial_\mu u_t (\mu,v)$ and $\partial^2_{\mu} u_t(\mu,v,v^\prime)$ respectively. Thus we can assume $u$ and its derivatives to be $\bP$-a.s. bounded. 

Again as in Theorem \ref{theo:IW-FullFlow-measureOnly} we consider $\mu \mapsto (u \star \rho)(\mu * \phi_G)$ instead of $\mu \mapsto (u \star \rho)(\mu)$ with $\phi_G$ - density of standard $d$-dimensional Gaussian distribution $N(0, I_d)$ on $\bR^d$ and $(\mu * \phi_G)(x):=\int_{\bR^d} \phi_G(x-y) d\mu(y)$. Now the support of $\mu *\phi_G$ is the whole $\bR^d$ and $\partial_\mu u, \partial_v\partial_\mu u$ and $\partial^2_{\mu} u$ are $\bP$-a.s. continuous at every triple $(t,\mu \circ \phi_G, v), t \in [0,T], v \in \bR^d $. Installing $\phi_{\varepsilon,G}$ - $d$-dimensional Gaussian distribution $N(0;\varepsilon I_d)$ and letting $\varepsilon \searrow 0$, we conclude the $\bP$-a.s. convergence of $\partial_\mu u_t (\mu*\phi_{\varepsilon,G},v)$, $\partial_v \partial_\mu u_t (\mu*\phi_{\varepsilon,G},v)$ and $\partial^2_{\mu} u_t(\mu*\phi_{\varepsilon,G},v,v^\prime)$ to $\partial_\mu u_t (\mu,v)$, $\partial_v \partial_\mu u_t (\mu,v)$ and $\partial^2_{\mu} u_t(\mu,v,v^\prime)$ respectively.
Thus we can assume $\bP$-a.s. uniform continuity of measure expansion terms for the whole $\bR^d$.

Now we are to show that mollification procedure is well-posed. 
It is straightforward to verify that $u_n:= u \star \rho_n$ satisfies $\bP$-a.s. \eqref{cond:Partial-Flow-IntRF-FullChainRule-U} uniformly in $n \geq 1 $.
Applying twice the dominated convergence theorem we conclude the $\bP$-a.s. convergence for all the terms but the stochastic integral. To handle the latter one additionally requires an argument across the quadratic variation, as written in Theorem \ref{theo:ExtendItoWentzell} and localisation.

As before we define $\phi_t=\psi^0_t=\psi^1_t := 0$, for those $t$ where the functions are not well-defined. We copy the procedure above to conclude that $\phi,\psi^0,\psi^1~\bP$-a.s. have  compact support. 

\emph{Step 2. Approximation.} By our mollification argument one can assume the $u, \partial_\mu u, \partial_v\partial_\mu u, \partial_x\partial_\mu$ and $\partial_\mu^2 u$ to be $\bP$-a.s. bounded and $\bP$-a.s. uniformly continuous in respective topology spaces.
We construct twin processes $(Y^l_t)_{t \in [0,T]}, \: l = 1,\dots, N$ of $(Y_t)_{t \in [0,T]}$ each supporting its own independent Brownian motion $(W^{1,l}_t)_{t\in[0,T]}$ that generate $(\Omega^{1,l},\cF^{1,l}, \bF^{1,l},\bP^{1,l})$ alongside with $\cF_0^l$, altogether forming a copy of $(\Omega^{1},\cF^{1}, \bF^{1},\bP^{1})$. Since the stochastic basis $(\Omega,\cF,\bF,\bP)$ of our initial space is constructed as a completion of $(\Omega^0 \times \Omega^1, \cF^0 \otimes \cF^1, \bF^0 \otimes \bF^1, \bP^0 \otimes \bP^1)$ augmented in a right-continuous way and then completed, we introduce a new product basis $(\Omega^{l},\cF^{l}, \bF^{l},\bP^{l}) $ to be completion of $(\Omega^0 \times \Omega^{1,l}, \cF^0 \otimes \cF^{1,l}, \bF^0 \otimes \bF^{1,l}, \bP^0 \otimes \bP^{1,l})$ augmented in right-continuous way and then completed. Now we copy the dynamics of $(Y_t)_{t \in [0,T]}$, as 
$$\dd Y_t^l = b_t^l \dd t + \sigma^{0,l}_t \dd W_t^0\dd t + \sigma_t^{1,l} \dd W^{1,l}_t, \quad Y^l_0 = Y_0^l,$$
where $Y_0^l, b^l_t, \sigma_t^{0,l}$ and $\sigma_t^{1,l}$ are copies of $Y^0,b_t, \sigma^0_t, \sigma_t^1$ respectively.
Now we construct a total stochastic basis $(\Omega^{1,\dots,N},\cF^{1,\dots,N}, \bF^{1,\dots,N},\bP^{1,\dots,N})$, where
\begin{align*}
    \Omega^{1,\dots, N} = \Omega^0 \times \Omega^1 \times \prod_{l=1}^N \Omega^{1,l}&, \quad  \cF^{1,\dots, N} = \cF^0 \otimes \cF^1 \otimes \bigotimes_{l=1}^N \cF^{1,l}, \\
    \bF^{1,\dots, N} = \bF^0 \otimes \bF^1 \otimes \bigotimes_{l=1}^N \bF^{1,l}&, \quad 
    \bP^{1,\dots, N} = \bP^0 \otimes  \bP^1 \otimes \bigotimes_{l=1}^N \bP^{1,l},
\end{align*}
where we again and finally augment the filtration in a right-continuous way and complete.
We underline that processes $\big((Y^l_t)(\omega^0,\cdot),b^l_t(\omega^0,\cdot),\sigma^{0,l}_t(\omega^0,\cdot),\sigma^{1,l}_t(\omega^0,\cdot)\big)_{t \in [0,T]},~ l = 1, \dots N$ are i.i.d. $\bP^0$-a.s.  

Hereinafter while fixing the $\omega^0\in \Omega^0$, and for the sake of simplicity we will omit adding the $(\omega^0,\cdot)$ to the processes $Y_t,~b_t,~\sigma^0_t,\sigma^1_t$ to highlight the respective relation to $\omega^0$, but will leave in after $\bar\mu^N_t$ as to underline the nature of this dependency.

Denoting the flow of marginals for almost all $\omega^0 \in \Omega^0$ as $\bar \mu_t^N(\omega^0,\cdot) := \frac1N\sum_{l = 1}^N \delta_{Y_t^l(\omega^0,\cdot)}$ for $t \in [0,T]$ and the empirical projection of $u$ as $u^N$ we proceed by applying It\^o-Wentzell formula (Theorem \ref{theo:ExtendItoWentzell}) to $u^N_t$ and using Proposition \ref{prop:DerivativeRelations-Space-2-Lions} to expand $\bP $-a.s.
\begin{align*}
    &u^N_t(Y^1_t,\dots,Y^N_t) - u^N_0(Y^1_0, \dots, Y^N_0) = \int_0^t \phi_s(\bar \mu^N_s(\omega^0,\cdot)) \dd s + \int_0^t \psi_s^{0}(\bar \mu^N_s(\omega^0,\cdot)) \cdot \dd W^0_s \\
    &+ \int_0^t\psi_s^{1}(\bar \mu^N_s(\omega^0,\cdot)) \cdot \dd W^{1}_s
    +   \frac{1}{N} \sum_{l=1}^N \int_0^t \partial_\mu u_s(\bar \mu_s^N(\omega^0,\cdot), Y_s^l) \cdot b_s^l \dd s \\
    &
    +   \frac{1}{N}\sum_{l=1}^N \int_0^t \partial_\mu u_s(\bar \mu_s^N(\omega^0,\cdot), Y_s^l) \cdot \sigma_s^{0,l} \dd W_s^{0}
    +   \frac{1}{N}\sum_{l=1}^N \int_0^t \partial_\mu u_s(\bar \mu_s^N(\omega^0,\cdot), Y_s^l) \cdot \sigma_s^{1,l} \dd W_s^{1,l} \\
    &
    +   \frac{1}{2N}  \sum_{l = 1}^N\int_0^t \trace\big\{\partial_v\partial_\mu u_s(\bar \mu_s^N(\omega^0,\cdot),Y_s^l) (\sigma_s^{0,l} (\sigma_s^{0,l})^\intercal+ \sigma_s^{1,l} (\sigma_s^{1,l})^\intercal)\big\}\dd s \\
    &
    +   \frac{1}{2N^2}  \sum_{l,l^\prime = 1}^{N,N} \int_0^t \trace\big\{\partial^2_{\mu} u_s(\bar \mu_s^N(\omega^0,\cdot),Y_s^l,Y_s^{l^\prime}) ~\sigma_s^{0,l} (\sigma_s^{0,l^\prime})^\intercal\big\}\dd s \\
    &
    +   \frac{1}{2N^2} \sum_{l = 1}^N \int_0^t \trace\big\{\partial^2_{\mu} u_s(\bar \mu_s^N(\omega^0,\cdot),Y_s^l,Y_s^l) ~\sigma_s^{1,l} (\sigma_s^{1,l})^\intercal\big\}\dd s \\
    &
    +   \frac{1}{N} \sum_{l = 1}^N\int_0^t \trace\big\{\partial_\mu \psi^0_s (\bar \mu^N_s(\omega^0,\cdot),Y_s^l) (\sigma^{0,l}_s)^\intercal\big\} \dd s.
\end{align*}

We highlight that we do not have $\partial_{\mu} \psi^1$ terms due to the fact that $\langle W^1,W^{1,l}\rangle_t = 0,~l = 1,\dots,N$, whilst one of the $\partial^2_\mu u$ terms is summed up diagonally, due to independence of $W^i, W^j,~ i,j \in 1,\dots,N,~ i\neq j$. 
Taking conditional expectations on the above formula $\bE^{1,1,\dots, N}\big[ \cdot\big] :=\bE^{\bP^{1,\dots,N}} \big[ \cdot |\: \cF^0 \otimes \cF^1\big]$ we have by the stochastic Fubini theorem (see \cite{Veraar2010FubiniRevisited}*{Theorem 3.5}) and boundedness of $\partial_\mu^2 u$ for any $t \in [0,T],\: \bP$-a.s.
\begin{align}
    \nonumber
    &\bE^{1,1,\dots, N}\Big[u_t(\bar\mu^N_T(\omega^0,\cdot)]\Big] - \bE^{1,1,\dots N}\Big[u_0(\bar\mu^N_0)\Big] = \bE^{1,1,\dots, N}\Big[\int_0^t \phi_s(\bar \mu^N_s(\omega^0,\cdot)) \dd s \Big] \\
    \nonumber
    &+ \bE^{1,1,\dots, N}\Big[\int_0^t \psi_s^{0}(\bar \mu^N_s(\omega^0,\cdot)) \cdot \dd W^0_s \Big]
    + \bE^{1,1,\dots, N}\Big[\int_0^t\psi_s^{1}(\bar \mu^N_s(\omega^0,\cdot)) \cdot  \dd W^{1}_s\Big]\\
    \nonumber
    &
    +
    \int_0^t \bE^{1,1,\dots N}\Big[\partial_\mu u_s(\bar \mu_s^N(\omega^0,\cdot), Y_s^1) \cdot b_s^1 \Big]\dd s \\
    \label{eq:PartialFlow-MeasureOnly-Approximation}
    &
    + \int_0^t \bE^{1,1,\dots N}\Big[(\sigma_s^{0,1})^\intercal\partial_\mu u_s(\bar \mu_s^N(\omega^0,\cdot), Y_s^1)   \Big] \cdot \dd W_s^{0}   \\
    \nonumber
    &
    +0+ \frac{1}{2}  \int_0^t\bE^{1,1,\dots N}\Big[ \trace\big\{\partial_v\partial_\mu u_s(\bar \mu_s^N(\omega^0,\cdot),Y_s^1) (\sigma_s^{0,1} (\sigma_s^{0,1})^\intercal+ (\sigma_s^{1,1})(\sigma_s^{1,1})^\intercal)\big\}\Big]\dd s\\
    \nonumber
    &
    +
    \frac{1}{2} \int_0^t \bE^{1,1,\dots N}\Big[ \trace\big\{\partial^2_{\mu} u_s(\bar \mu_s^N(\omega^0,\cdot),Y_s^1,Y_s^{2}) ~\sigma_s^{0,1} (\sigma_s^{0,2})^\intercal\big\} \Big]\dd s\\
    \nonumber
    &
    +
    \int_0^t \bE^{1,1,\dots N}\Big[\trace\big\{\partial_\mu \psi^0_s (\bar \mu^N_s(\omega^0,\cdot),Y_s^1) (\sigma^{0,1}_s)^\intercal\big\}\Big] \dd s
    + O\big(\frac{1}{N}\big),
\end{align}
with $O(1/N)$ standing for the Bachmann-Landau big-$O$ notation (sequence bounded by $\frac{C}{N},$ for some $C \geq 0$) which appears from the second $1/N^2$ summation term (notice the sum is over only one index).

Lifting to  $L_2$-space and using continuity of the underlying process, as in \cite{CarmonaDelarue2017book2}*{Theorem 4.14}, we conclude that $\bP^0 \otimes \bP^1$-a.s. 
$$\limsup_{N \to \infty} \bE^{1,1,\dots N}\Big[\sup_{0 \leq s \leq T}
W_2\big( \bar\mu^N_s(\omega^0,\cdot),\mu_s(\omega^0,\cdot)\big)^2\Big] = 0.$$

Now due to the continuity in the measure-component and dominated convergence theorem we can pass to the limit in  \eqref{eq:PartialFlow-MeasureOnly-Approximation} (as $N \to \infty)$ to conclude the formula. The convergence of stochastic integral is secured by localisation and arguing across quadratic variation. We swap the integral and expectation by stochastic Fubini theorem. Finally we rewrite  the expectations in the RHS upon dependance on two model particles (living on $(\Omega^0 \times \tilde\Omega^1, \cF^0 \otimes\tilde\cF^1, \bF^0 \otimes \tilde\bF^1, \bP^0 \otimes \tilde\bP^1)$ and $(\Omega^0 \times \hat\Omega^1, \cF^0 \otimes\hat\cF^1, \bF^0 \otimes \hat\bF^1, \bP^0 \otimes \hat\bP^1)$ respectively).
Measurability is again secured by Remark \ref{rem:measurability-main-remark}.
\end{proof}

\subsection{The joint chain rule}

Now we are ready to prove a joint chain rule for $u: \Omega \times [0,T] \times\bR^d \times\cP_2(\bR^d) \to \bR$ as given by \eqref{eq:RandomField-2BM}.
We introduce minimal regularity requirements.

\begin{definition}
\label{def:PartialC2-RandomField-PartialFlow-ExtendedChainRule}
We say the random field $u:\Omega\times [0,T] \times \bR^d \times \cP_2(\bR^d)\to \bR$ given in \eqref{eq:RandomField-2BM} is \emph{RF-Joint-Generally-$\cC^{2}$} if 
\begin{enumerate}[i)]
    \item $u$ is \emph{RF-Joint-Partially}-$\cC^2$ for $\psi_t := (\psi_t^0,\psi_t^1)^\intercal$ and $W_t := (W^0_t,W^1_t)^\intercal$;

    \item For almost any $t \in[0,T]$, the maps $(x,\mu)\mapsto \phi_t(x,\mu),~(x,\mu)\mapsto \psi^0_t(x,\mu),~(x,\mu)\mapsto \psi^1_t(x,\mu)$, are $\bP$-a.s.~joint-continuous in product topology of $\bR^d \times \cP_2(\bR^d)$ at every pair $(x,\mu)\in \bR^d \times \cP_2(\bR^d)$;
    
    \item For almost any $t \in [0,T], ~ x\in \bR^d$, the map $\mu \mapsto \psi^0_t(x,\mu)$ is $L$-differentiable $\bP$-a.s.~at every point $\mu \in \cP_2(\bR^d)$. Moreover, $\partial_\mu \psi^0_t(x,\mu,v)$ has a $\mu$-version such that $\partial_\mu \psi^0_t(x,\mu,v)$ is $\bP$-a.s. joint-continuous at every pair $(\mu,v),~\mu \in \cP_2(\bR^d),~v \in \Supp(\mu),~\bP$-a.s.;
    
    \item For any $(t,\mu,v) \in [0,T] \times \cP_2(\bR^d) \times \Supp(\mu)$, the map $x \mapsto \partial_\mu u_t(x,\mu,v)$ is $\bR^d$-differentiable $\bP$-a.s. at every point $x \in \bR^d$. Moreover, $\partial_x\partial_\mu u_t(x,\mu,v)$ has a $\mu$-version such that $\partial_x \partial_\mu u_t(x,\mu,v)$ is $\bP$-a.s. joint-continuous at every quadruple $(t,x,\mu,v),$ with $ (t,x,\mu) \in [0,T] \times \bR^d \times \cP_2(\bR^d),~v \in \Supp(\mu),~\bP$-a.s.;
    
    \item For any $(t,x,v) \in [0,T]\times \bR^d \times \Supp(\mu)$, the map $\mu \mapsto \partial_\mu u_t(x,\mu,v)$ is L-differentiable $\bP$-a.s. at every point $ \mu \in \cP_2(\bR^d)$. Moreover, $\partial^2_{\mu} u_t(x,\mu,v,v^\prime)$ has a $\mu$-version such that $\partial^2_\mu u_t(x,\mu,v,v^\prime)$ is $\bP$-a.s. joint-continuous at every quintuple $(t,x,\mu,v,v^\prime),$ with $ (t,x,\mu) \in [0,T] \times \bR^d \times \cP_2(\bR^d),~v,v^\prime \in \Supp(\mu),~\bP$-a.s..
\end{enumerate}
\end{definition}
We highlight again the slight abuse of notation in the way point i) in the above Definition \ref{def:PartialC2-RandomField-PartialFlow-ExtendedChainRule} is formulated. This avoids re-stating a full assumption that is nonetheless clear to understand.

\begin{theorem}
\label{theo:IW-PartialMeasureFlow-Proc+Measure-2BM}
Let $u$ be \textit{RF-Joint-Generally}-$\cC^{2}$ It\^o random field \eqref{eq:RandomField-2BM}. Assume for any compact $K\subset \bR^d \times\cP_2(\bR^d)$ that
\begin{align*}
	\int_0^T \sup_{(x,\mu) \in K}\bigg\{|\phi_s(x,\mu)| 
	&+ |\psi^0_s(x,\mu)|^2 
	+ |\partial_x \psi^0_s(x,\mu)|^2 + \int_{\bR^d}  |\partial_\mu \psi^0_s(x,\mu,v)|^2 \mu(\dd v)\\
	&+|\psi^1_s(x,\mu)|^2 
	+ |\partial_x \psi^1_s(x,\mu)|^2\bigg\}\dd s < \infty, \quad \bP\text{-a.s.},
\end{align*}
and
\begin{align*}
	\sup_{(t,x,\mu) \in [0,T] \times K}
	\Big\{ &\int_{\bR^d}\Big[|\partial_\mu u_t(x,\mu, v)|^2
	+ |\partial_v\partial_\mu u_t(x,\mu, v)|^2
	+|\partial_x\partial_\mu u_t(x,\mu, v)|^2\Big]\mu(\dd v) 
	\\
	&+\int_{\bR^d \times \bR^d}|\partial^2_{\mu} u_t(\mu,v,v^\prime)|^2 \mu(\dd v)\mu(\dd v^\prime) \Big\} < \infty, \quad \bP\text{-a.s.}.
\end{align*}
Take $(\mu_t)_{t \in [0,T]} = \big(\Law(Y_t(\omega^0,\cdot))\big)_{t \in [0,T]}$ with $(Y_t)_{t \in [0,T]}$ solution to \eqref{eq:GenericYprocess-PartialFlow22} under Assumption \ref{Assump:SDE-Y-mu-2BM} and $(X_t)_{t \in [0,T]}$ given by \eqref{eq:ProcessX-partialFlowMeasures} under Assumption \ref{Assump:SDE-X-2BM}.

Then $(u_t(X_t,\mu_t))_{t\in[0,T]}$ is an It\^o process $\bP$-a.s. satisfying the expansion 
\begin{align}
\nonumber
u_T(X_T&,\mu_T) - u_0(X_0,\mu_0) 
= 
\int_0^T\phi_s(X_s,\mu_s) \dd s + \int_0^T\psi^0_s(X_s,\mu_s) \cdot \dd W^0_s+ \int_0^T\psi^1_s(X_s,\mu_s) \cdot \dd W^1_s
\\
\nonumber
&
+ \int_0^T \partial_xu_s(X_s,\mu_s)  \cdot\beta_s \dd s 
+ \int_0^T \partial_xu_s(X_s,\mu_s)  \cdot\gamma^0_s \dd W^0_s
+ \int_0^T \partial_xu_s(X_s,\mu_s)  \cdot\gamma^1_s \dd W^1_s \\
\nonumber
&
+ \int_0^T \frac{1}{2}~\trace\big\{\partial_{xx}^2 u_s(X_s,\mu_s)  (\gamma^0_s (\gamma^0_s)^\intercal  +  \gamma^1_s (\gamma^1_s)^\intercal)\big\} \dd s
\\
\nonumber
&
+   \int_0^T \tilde{\bE}^1 \Big[\partial_\mu u_s(X_s,\mu_s,\tilde{Y}_s)\cdot\tilde{b}_s\Big]\dd s 
+   \int_0^T \tilde{\bE^1} \Big[(\tilde{\sigma}^0_s)^\intercal\partial_\mu u_s(X_s,\mu_s,\tilde{Y}_s)\Big]\cdot \dd W_s^0 \\
\label{eq:Partial-Flow-Full-Chain-Rule}
&
+   \int_{0}^{T}\frac{1}{2}~\tilde{\bE}^1\Big[  \trace \big\{\partial_v\partial_\mu u_s(X_s,\mu_s, \tilde{Y}_s)  (\tilde{\sigma}^0_s(\tilde{\sigma}_s^0)^{\intercal}+\tilde{\sigma}^1_s(\tilde{\sigma}_s^1)^{\intercal})\big\}\Big] \dd s
\\
\nonumber
&
+\int_0^T \frac12~\hat{\bE}^1 \Big[\tilde{\bE}^1 \Big[ \trace \big\{\partial^2_{\mu} u_s(X_s,\mu_s,\tilde{Y}_s,\hat{Y}_s) ~\tilde{\sigma}^0_s(\hat \sigma_s^0)^\intercal\big\} \Big]\Big] \dd s \\
\nonumber
&
+ \int_{0}^{T}\tilde{\bE}^1\Big[  \trace \big\{\partial_x\partial_\mu u_s(X_s,\mu_s, \tilde{Y}_s) ~\gamma^0_s(\tilde{\sigma}_s^0)^{\intercal}\big\}\Big] \dd s
+ \int_0^T \trace \big\{\partial_x\psi^0_s(X_s,\mu_s) (\gamma^0_s)^\intercal \big\} \dd s\\
\nonumber
&
+ \int_0^T \trace \big\{\partial_x\psi^1_s(X_s,\mu_s) (\gamma^1_s)^\intercal \big\} \dd s
+ \int_{0}^{T}\tilde{\bE}^1\Big[\trace\big\{ \partial_\mu \psi^0_s(X_s,\mu_s,\tilde{Y}_s) (\tilde{\sigma}^0_s)^\intercal\big\}\Big] \dd s,
\end{align}
where the formula above $\tilde{\bE}$ and $\hat{\bE}$ denote the expectation acting on the model twin spaces $(\tilde{\Omega}, \tilde{\bF}, \tilde{\bP})$ and $(\hat{\Omega}, \hat{\bF}, \hat{\bP})$ respectively, and let the processes $(\tilde{Y}_t, \tilde{b}_t, \tilde{\sigma}_t)_{t\in[0,T]}$ and $(\hat{Y}_t, \hat{b}_t, \hat{\sigma}_t)_{t\in[0,T]}$ be the independent twin processes of $(Y_t,b_t,\sigma_t)_{t\in[0,T]}$ respectively living within.
\end{theorem}

It is interesting to mention the $\partial_x \partial_\mu u$ term, which also appears in Theorem \ref{theo-Classic-Ito-Lions-PartialFlow}, it is nothing else but the cross-variation of the process $X$ and the model particle $\tilde Y$. The very last two lines contain all the possible ways of cross-interactions, namely, between the random field $u$, the process $X$ and the random measure $\mu$.

\begin{remark}
According to \cite{crisan2018smoothing}, $\partial_x\partial_\mu u=\partial_\mu\partial_x u$, when both crossed derivatives exist and are Lipschitz.  However, as one can notice in the proof, within our mollification procedure for empirically projected mapping the space derivatives could be swapped in the convenient way to secure the existence of the limit - desired derivative. Thus in the Definition \ref{def:PartialC2-RandomField-PartialFlow-ExtendedChainRule} one can equally demand the existence of $\partial_\mu \partial_x u$ instead of $\partial_x \partial_\mu$. The same applies for respective derivatives for $\phi, \psi^0$ and $\psi^1$.
\end{remark}

The proof used in Theorem \ref{theo:IW-FUllmeasureFlow-Proc+Measure-1BM} does not carry directly to this case, crucially due to the passage to the limit in \eqref{eq:Convergence-Of-ForwardMeasureFlow} as the measure flow is random. Of the possible angles of attack to show the result the direct application of the empirical projection approach is the simplest. We follow it and provide alternative arguments when the passage to the limit issue arises. 

\begin{proof}[Proof of Theorem \ref{theo:IW-PartialMeasureFlow-Proc+Measure-2BM}]

In view of the proof of Theorem \ref{theo:IW-PartialFlow-measureOnly} we assume a compactification \& mollification argument in the measure component as been applied and hence we do not repeat its construction. Moreover, without loss of generality assume $(b_t)_{t \in [0,T]}$ and $(\sigma_t)_{t \in [0,T]}$ to be bounded.

Again as in previous theorem, we consider $u^N$- empirical projection of $u$, and construct generic $(Y^l_t)_{t \in [0,T]}$ in the same way, underlining that the processes $Y^l_t(\omega^0,\cdot)$, $b^1_t(\omega^0,\cdot)$, $\sigma^0_t(\omega^0,\cdot)$,  $\sigma^1_t(\omega^0,\cdot)$, ${t \in [0,T]}$, $l = 1, \dots N$ are i.i.d. 
For $\phi_t,~\psi_t :=  \left( \begin{smallmatrix} \psi_t^0 & 0 \\ 0 & \psi_t^1 \end{smallmatrix} \right)$ and $W_t := (W^0_t,W^1_t)^\intercal$ we copy the same procedure as before to have for almost all $t$, $\bP$-a.s.
\begin{align*}
    \bE^{1,1,\dots,N}\big[\phi^N_t(X_t,Y^1_t,\dots, Y^N_t)\big] &
    \to \phi_t(X_t,\mu_t)
    \quad\textrm{and}\quad
    \bE^{1,1,\dots,N}\big[\psi^N_t(X_t,Y^1_t,\dots, Y^N_t)\big] 
    \to \psi_t(X_t,\mu_t),
\end{align*}
as $N\to \infty$, together with
\begin{align*}
    &\bE^{1,1,\dots,N}\Big[\int_0^t\phi^N_s(X_s,Y^1_s,\dots, Y^N_s)\dd s\Big] \to \int_0^t\phi_s(X_s,\mu_s)\dd s,\\
    &\bE^{1,1,\dots,N}\Big[\int_0^t\psi^N_s(X_s,Y^1_s,\dots, Y^N_s)\cdot \dd W_s\Big] \to \int_0^t\psi_s(X_s,\mu_s)\cdot \dd W_s,
\end{align*}
$\bP$-a.s. as $N \to \infty$ for all $t \in [0,T]$. As before, for the sake of simplicity we omit adding the $(\omega^0,\cdot)$ to the processes $Y_t,~b_t,~\sigma^0_t,\sigma^1_t$, but will leave one for $\bar\mu^N_t$.

Since all conditions of Theorem \ref{theo:ExtendItoWentzell} hold we apply it to $u^N_t(X_t,Y^1_t,\dots,Y^N_t)$ getting
\begin{align*}
    u^N(&X_T,Y_T^1,\dots,Y_T^N) - u^N(X_0,Y_0^1\dots,Y_0^N) 
    = \int_0^T\phi^N_s(X_s,Y_s^1,\dots,Y_s^N)\dd s\\
    & 
    + \int_0^T\psi^{0,N}_s(X_s,Y_s^1,\dots,Y_s^N)\cdot \dd W^0_s
    + \int_0^T\psi^{1,N}_s(X_s,Y_s^1,\dots,Y_s^N)\cdot \dd W_s^1\\
    & 
    + \int_0^T\partial_xu^{N}_s(X_s,Y_s^1,\dots,Y_s^N)\cdot\beta^1_s\dd s
    + \int_0^T\partial_xu^{N}_s(X_s,Y_s^1,\dots,Y_s^N)\cdot\gamma^0_s\dd W_s^0\\
    &
    + \int_0^T\partial_xu^{N}_s(X_s,Y_s^1,\dots,Y_s^N) \cdot\gamma^1_s\dd W_s^1
    + \frac12\int_0^T\trace\big\{\partial_{xx}u^{N}_s(X_s,Y_s^1,\dots,Y_s^N)~ \gamma^0_s(\gamma^0_s)^\intercal\big\}\dd s\\
    &
    +\frac12 \int_0^T\trace\big\{\partial_{xx}u^{N}_s(X_s,Y_s^1,\dots,Y_s^N)~ \gamma^1_s(\gamma^1_s)^\intercal\big\}\dd s\\
    &
    + \int_0^T  \trace\big\{\partial_x\psi^{0,N}_s(X_s,Y_s^1,\dots,Y_s^N) (\gamma_s^0)^\intercal \big\}\dd s
    + \int_0^T  \trace\big\{\partial_x\psi^{1,N}_s(X_s,Y_s^1,\dots,Y_s^N) (\gamma_s^1)^\intercal \big\}\dd s\\
    & 
    + \sum_{l=1}^N \int_0^T \partial_{y^l} u^{N}_s(X_s,Y_s^1,\dots,Y_s^N) \cdot\beta^1_s\dd s
    + \sum_{l=1}^N \int_0^T \partial_{y^l} u^{N}_s(X_s,Y_s^1,\dots,Y_s^N) \cdot\sigma^{0,l}_s \dd W^0_s\\
    &
    + \sum_{l=1}^N \int_0^T \partial_{y^l} u^{N}_s(X_s,Y_s^1,\dots,Y_s^N) \cdot\sigma^{1,l}_s \dd W^{1,l}_s\\
    &
    + \frac12 \sum_{l,l^\prime=1}^{N,N}\int_0^T \trace\big\{\partial_{y^ly^{l^\prime}} u^{N}_s(X_s,Y_s^1,\dots,Y_s^N) ~ \sigma^{0,l}_s(\sigma^{0,l^\prime}_s)^\intercal\big\}\dd s \\
    &
    + \frac12\sum_{l=1}^{N}\int_0^T \trace\big\{\partial_{y^ly^{l^\prime}} u^{N}_s(X_s,Y_s^1,\dots,Y_s^N)~ \sigma^{1,l}_s(\sigma^{1,l}_s)^\intercal\big\}\dd s\\
    &
    + \sum_{l=1}^{N}\int_0^T \trace\big\{\partial_{xy^l} u^{N}_s(X_s,Y_s^1,\dots,Y_s^N)  ~\gamma^0_s(\sigma^{0,l}_s)^\intercal\big\} \dd s\\
    &
    + \sum_{l = 1}^N\int_0^T \trace\big \{\partial_{y^l} \psi^{0,N}_s(X_s,Y_s^1,\dots,Y_s^N) (\sigma_s^{0,l})^\intercal \big\}\dd s.
\end{align*}
We again underline that we do not have additional $\partial_{xy^l}u$ and $\partial_{y^l} \psi$ terms due to the fact that $\langle W^1,W^{1,l}\rangle_t = 0,~ l = 1,\dots,N$, at the same time diagonally summing one of $\partial^2_\mu u$, due to mutual independence of $W^i, W^j,~i,j \in 1,\dots,N,~ i\neq j$.

Now we transform the equation according to Proposition \ref{prop:DerivativeRelations-Space-2-Lions}, and applying $\bE^{1,1,\dots,N}\big[\cdot\big]:=\bE\big[\cdot|\cF^0\otimes \cF^1\big]$, law of large numbers, Fubini theorem and boundedness of $\partial_\mu^2 u$ we get $\bP$-a.s
\begin{align*}
    \bE&^{1,1,\dots,N}\big[u(X_T,\bar\mu_T^N)\big] - \bE^{1,1,\dots,N}\big[u^N(X_0,\bar \mu_0^N) \big]
    = \int_0^T\bE^{1,1,\dots,N}\Big[\phi^N_s(X_s,\bar \mu^N_s(\omega^0,\cdot))\Big]\dd s\\
    &
    + \int_0^T\bE^{1,1,\dots,N}\Big[\psi^{0,N}_s(X_s,\bar \mu^N_s(\omega^0,\cdot))\Big]\cdot \dd W^0_s
    + \int_0^T\bE^{1,1,\dots,N}\Big[\psi^{1,N}_s(X_s,\bar \mu^N_s(\omega^0,\cdot))\Big]\cdot \dd W_s^1 \\
    &
    + \int_0^T\bE^{1,1,\dots,N}\Big[\partial_xu_s(X_s,\bar \mu^N_s(\omega^0,\cdot))\Big] \cdot \beta^1_s\dd s
    + \int_0^T\bE^{1,1,\dots,N}\Big[\partial_xu_s(X_s,\bar \mu^N_s(\omega^0,\cdot))\Big] \cdot \gamma^0_s\dd W_s^0\\
    &
    + \int_0^T\bE^{1,1,\dots,N}\Big[\partial_xu_s(X_s,\bar \mu^N_s(\omega^0,\cdot))\Big] \cdot \gamma^1_s\dd W_s^1\\
    &
    + \int_0^T\frac12~\bE^{1,1,\dots,N}\Big[\trace\big\{\partial_{xx}u_s(X_s,\bar \mu^N_s(\omega^0,\cdot))(\gamma^0_s(\gamma^0_s)^\intercal+ \gamma^1_s(\gamma^1_s)^\intercal)\big\}\Big] \dd s\\
    &
    + \int_0^T\bE^{1,1,\dots,N}\Big[ \partial_x \psi^{0}_s(X_s,\bar \mu^N_s(\omega^0,\cdot)) \Big]\gamma_s^0\dd s
     + \int_0^T\bE^{1,1,\dots,N}\Big[ \partial_x \psi^{1}_s(X_s,\bar \mu^N_s(\omega^0,\cdot))\Big]\gamma_s^1\dd s\\
     & 
     + \int_0^T\bE^{1,1,\dots,N}\Big[ \partial_\mu u_s(X_s,\bar \mu^N_s(\omega^0,\cdot),Y_s^1) \cdot \beta^1_s\Big]\dd s
     + \int_0^T\bE^{1,1,\dots,N}\Big[ (\sigma^{0,1}_s)^\intercal \partial_\mu u_s(X_s,\bar \mu^N_s(\omega^0,\cdot)) \Big]\cdot \dd W^0_s\\
     &
     + \int_0^T\frac12~\bE^{1,1,\dots,N}\Big[ \trace\big\{\partial_{v}\partial_\mu u_s(X_s,\bar \mu^N_s(\omega^0,\cdot),Y_s^1)  (  \sigma^{0,1}_s(\sigma^{0,1}_s)+\sigma^{1,1}_s(\sigma^{1,1}_s)^\intercal)\big\}\Big]\dd s \\
     &
     + \int_0^T\frac12~\bE^{1,1,\dots,N}\Big[ \trace\big\{\partial_\mu^2u_s(X_s,\bar \mu^N_s(\omega^0,\cdot),Y_s^1,Y_s^2) ~\sigma^{0,1}_s(\sigma^{0,2}_s)^\intercal\big\}\Big]\dd s\\
     &
     + \int_0^T~\bE^{1,1,\dots,N}\Big[ \trace\big \{\partial_{x}\partial_\mu u_s(X_s,\bar \mu^N_s(\omega^0,\cdot),Y_s^1)  ~\gamma^0_s(\sigma^{0,1}_s)^\intercal\big\}\Big] \dd s\\
     &
     + \int_0^T\bE^{1,1,\dots,N}\Big[ \trace\big \{\partial_\mu \psi^{0}_s(X_s,\bar \mu^N_s(\omega^0,\cdot),Y^1_s) (\sigma_s^{0,1})^\intercal\big\}\Big]\dd s + O\big(\frac1N\big).
\end{align*}
We note that the expectation taken on the term in the fifth line does not charge the process $(\gamma^0(\gamma^0)^\intercal+\gamma^1(\gamma^1)^\intercal)$, we write it as it is to preserve the matrix-trace notation.

According to the conditional propagation of chaos argument, as given in Theorem \ref{theo:IW-PartialFlow-measureOnly},
dominated convergence theorem (twice for the terms from the last five lines), localisation for $X$ and joint continuity and integrability of involved terms one can conclude the convergence of the above formula to \eqref{eq:Partial-Flow-Full-Chain-Rule}.
We argue additionally across convergence of quadratic variation to handle the stochastic integral terms. 

As before we switch to two model particles (living on $(\Omega^0 \times \tilde\Omega^1, \cF^0 \otimes\tilde\cF^1, \bF^0 \otimes \tilde\bF^1, \bP^0 \otimes \tilde\bP^1)$ and $(\Omega^0 \times \hat\Omega^1, \cF^0 \otimes\hat\cF^1, \bF^0 \otimes \hat\bF^1, \bP^0 \otimes \hat\bP^1)$ respectively) and swap the integral and expectation by stochastic Fubini theorem. Again and finally, we assert the measurability of involved terms by Remark \ref{rem:measurability-main-remark}.

\end{proof}

\textbf{Conflict of interests.} The authors declare that they have no conflict of interest.

%
%

%
%


%
%

%
%
%



\end{document}